\documentclass[11pt,onecolumn]{article}
\usepackage{amscd}
\usepackage{times}
\usepackage{amsmath}
\usepackage{amssymb}
\usepackage{xspace}
\usepackage{theorem}
\usepackage{graphicx}
\usepackage{ifpdf}
\usepackage{url,hyperref}
\usepackage{latexsym}
\usepackage{euscript}
\usepackage{xspace}
\usepackage[all]{xy}
\usepackage{color}
\usepackage{makeidx}
\usepackage{tikz}
\long\def\remove#1{}
\newtheorem{theorem}{Theorem}[section] 

\newtheorem{obs}[theorem]{Observation}
\newtheorem{corollary}[theorem]{Corollary}

\newtheorem{definition}[theorem]{Definition}
\newtheorem{proposition}[theorem]{Proposition}
\newenvironment{proof}{{\em Proof:}}{\hfill{\hfill\rule{2mm}{2mm}}}

\newcommand {\mm}[1] {\ifmmode{#1}\else{\mbox{\(#1\)}}\fi}

\newcommand{\img}{\mathrm img}
\newcommand{\supp}{\mathrm supp}
\newcommand{\coker} {\mathrm coker}

\newcommand{\rank}                {\mm {\rm rank}}

\newcommand{\cancel}[1]

\begin{document}

\title{A refinement of Betti numbers and homology in the  presence of a continuous function II
(the case of an angle valued map)}

\author{
Dan Burghelea  \thanks{
Department of Mathematics,
The Ohio State University, Columbus, OH 43210,USA.
Email: {\tt burghele@math.ohio-state.edu}}
}
\date{}

\date{}
\maketitle

\begin{abstract}
\vskip .1in 
For $f:X\to\mathbb S^1$ a continuous angle-valued map defined on a compact ANR $X,$  $\kappa$ a field and any integer $r\geq 0$ one proposes a refinement $\delta^f_r$ 
of the Novikov-Betti numbers   of  the pair $(X,\xi_f)$ and a refinement $\hat \delta^f_r$ of the Novikov homology of $(X,\xi_f)$ where $\xi_f$ denotes the integral degree one cohomology class represented by $f.$
The refinement $\delta^f_r$ is a configuration of points with multiplicity  located in $\mathbb R^2/\mathbb Z$ identified to the $\mathbb C\setminus 0$ whose total cardinality is the $r-$th Novikov--Betti number of the pair. The refinement $\hat \delta^f_r$ is a configuration of submodules of the $r-$th Novikov homology 
whose direct sum is isomorphic to the Novikov homology and with the same support as of $\delta^f_r.$
When $\kappa= \mathbb C$ the configuration $\hat\delta^f_r$ is convertible into a configuration of mutually orthogonal closed Hilbert submodules  of the $L_2-$homology of the infinite cyclic cover of $X$ defined by $f$ which is an $L^\infty(\mathbb S^1)-$Hilbert module.
One discusses the properties of these configurations namely, robustness with respect to continuous perturbation of the angle-values map and the Poincar\'e Duality
and one derives some computational applications in topology. The main results parallel the results for the case of real-valued map but with Novikov homology and Novikov Betti numbers replacing standard homology and standard Betti numbers.    
\end{abstract}

\thispagestyle{empty}
\setcounter{page}{1}

\maketitle

\setcounter{tocdepth}{1}
\tableofcontents

\section {Introduction } \label {I}

This paper is a sequel of \cite {B1} (which considers the case of real-valued map) 
but can be read independently of \cite{B1}.  Here we treat  the case of an angle-valued continuous map $f:X\to \mathbb S^1$ and complete results from \cite {BH}.  
In this paper without any additional specifications  an angle-valued map assumes that 
 the space $X$ is a compact ANR  (in particular a space homeomorphic to a finite simplicial complex or a compact Hilbert cube manifold) and the map $f$ is continuous.
 The map $f$ determines a degree one integral cohomology class $\xi_f \in H^1(X;\mathbb Z).$ 

We  fix a field $\kappa$  and an integer $r,$ $r=0,1,2, \dim X,$   and provide first
a configurations $\delta^f_r$ of finitely many points  with specified multiplicity located in the space $\mathbb T:= \mathbb R^2/\mathbb Z$ \footnote { $\mathbb R^2$ is equipped with the 
action $ \mu (n, (a,b)) \to (a+2\pi n, b+2\pi n).$} which can be  identified to the punctured plane  $ \mathbb C\setminus 0.$  It will be shown that  the set of points (counted with multiplicity) 
of this configuration has cardinality equal to the Novikov--Betti  number $\beta^N_r(X;\xi_f).$  In view of the identification of $\mathbb T$ with $\mathbb C\setminus 0$   the configuration 
$\delta^f_r$ can be also interpreted as  a monic \footnote{The monomial of highest degree has coefficient 1.} polynomial $P^f_r(z),$ with complex coefficients and nonzero free term of 
degree equal to the Novikov--Betti number, whose  roots are the points of the configuration $\delta^f_r.$  
  
We refine the configuration $\delta^f_r$ to the configuration $\hat \delta^f_r$ of $\kappa[t^{-1}, t]-$free modules indexed by $z\in \mathbb C\setminus 0$, each one a quotient of split free 
submodules of the $r-$th Novikov homology of $(X;\xi_f),$ 
and 
in case $\kappa= \mathbb C$ to the configuration
$\hat{\hat\delta}^f_r$ of closed Hilbert submodules of $L_2-$homology of $\tilde X,$ the infinite cyclic cover of $X$ defined by $\xi_f.$  
All configurations $\delta^f_r, \hat \delta^f_r, \hat {\hat \delta}^f_r$ are maps with finite support defined on $\mathbb C \setminus 0$ with $\delta^f_r(z)$ a nonpositive integer, $\hat \delta^f_r(z)$ 
a free $\kappa[t, t^{-1}]-$module and $\hat {\hat \delta}^f_r(z)$ a $L^\infty(\mathbb S^1)-$Hilbert module with ${\hat \delta}^f_r(z)$ resp. $\hat {\hat \delta}^f_r(z)$ of rank resp. von-Neumann 
dimension equal to $\delta^f_r(z).$

In this paper, for a fixed field $\kappa,$ the Novikov homology  $H^N_r(X;\xi)$ is a free $\kappa[t^{-1},t]-$module.   Novikov \cite{N} and most of the authors \cite {Pa} regard Novikov homology
as a $\kappa[t^{-1},t]]-$vector space, where $\kappa[t^{-1}, t]]$ denotes the field of Laurent power series with coefficients in $\kappa$  obtained by {\t extending the scalars} from 
$\kappa[t^{-1},t]$ to $\kappa[t^{-1},t]]$\footnote {The vector space is  equal to $H^N(X;\xi)\otimes_{\kappa[t^{-1},t]} \kappa[t^{-1},t]].$}. If 
the Novikov homology is regarded as a vector space over this field 
then $\hat \delta^ f_r$ is a configuration of vector subspaces, and this is entirely in analogy with the case of real-valued map treated in \cite{B1}.

The results about the configurations $\delta^f_r$ $\hat \delta^f_r$ and $\hat{\hat \delta}^f_r$  are formulated in Theorems \ref{TT1}, \ref {TT2} and \ref{TT3}
and are formally similar to Theorems 4.1, 4.2, and 4.3  in \cite{B1}, but  conceptually more complex  and technically more difficult to conclude.  In comparison with \cite{B1} there are however 
a number of differences and new features which deserve to be pointed out.
\begin{itemize}
\item  The location of the points in the support of the configurations $\delta^f_r, \hat \delta^f_r, \hat{\hat\delta} ^f_r$ is the space $\mathbb T:= \mathbb R^2/ \mathbb Z$ identified to
the punctured complex plane $\mathbb C\setminus 0$ by the map $\mathbb T\ni \langle a,b \rangle  \to z= e^{ia + (b-a)} \in \mathbb C\setminus 0$
and not $\mathbb R^2= \mathbb C$ as in \cite{B1}.
\item The Betti numbers $\beta_r (X)$ in \cite{B1}  are  replaced by the Novikov-Betti numbers $\beta^N_r(X;\xi)$ or by $L_2-$ Betti numbers $\beta^{L_2}_r(\tilde X),$  $\tilde X$ the infinite 
cyclic cover defined by $\xi= \xi_f,$ and the  homology  $H_r(X)$ is replaced by the Novikov homology of $(X, \xi)$ or by the $L_2-$homology of $\tilde X.$  
\item For $z= \langle a,b\rangle \in \supp \ \delta^f_r$ the configuration $\hat{\delta}^f_r$ takes as value $ \hat \delta^f_r\langle a, b\rangle= \hat {\delta}^f_r(z),$ a free $\kappa[t^{-1}, t]-$module 
which is a quotient $\hat {\mathbb F}_r(z) / \hat { \mathbb F}_r '(z)$ of split free submodules   
$\hat {\mathbb F}_r'(z) \subseteq \hat {\mathbb F}_r(z)\subseteq H_r^N(X;\xi).$ The configuration  $\hat\delta^f_r$ is derived from a  {\it  configuration of pairs of submodules} of $H_r^N(X,\xi_f),$ 
$\tilde\delta^f_r(z):= (\hat {\mathbb F}_r(z) , \hat {\mathbb F}_r'(z)),$  a concept  explained in section \ref{S2}.
\item In case $\kappa= \mathbb C,$ the ring of Laurent polynomials $\mathbb C[t^{-1}, t]$ has a natural completion to the finite von Neumann algebra $L^\infty(S^1)$ and $H_r^N(X;\xi_f)$ to a 
$L^\infty(\mathbb S^1)-$Hilbert module.
The Hilbert module structure, although unique up to isomorphism, depends on a chosen $\mathbb C[t^{-1}, t]-$compatible Hermitian inner product on $H^N_r(X;\xi_f),$ cf. section \ref{S2},  which 
always exists. With respect to a given $\mathbb C[t^{-1},t]-$compatible inner product  
the free module $H^N_r(X;\xi)$ can be  canonically converted into the $L^\infty(\mathbb S^1)-$Hilbert module $H^{L_2}_r(\tilde X),$ $\tilde X$ the infinite cyclic cover associated with $\xi,$ and  
the configuration $\hat \delta^f_r(z)$  into a configuration 
of mutually orthogonal closed Hilbert submodules  $\hat {\hat \delta}^f_r(z)\subseteq H^{L_2  }_r(\tilde X)$ with 
$\sum_{z\in \supp \delta^f_r}  \hat {\hat \delta}^f_r(z)= H^{L_2}_r(\tilde X).$ 
This conversion is referred below as the {\it von Neumann completion} and is described in section \ref{S2}.  A Riemannian metric  on $X,$ when $X$ is the underlying space of a closed smooth 
manifold, or a triangulation of $X,$ when $X$ is homeomorphic to a finite simplicial complex  
provides a canonical  inner product which leads to the  familiar $L_2-$homology,  $H^{L_2}_r(\tilde X).$  This is a particular case of a construction described in \cite {Lu97}. 
\item  The  refinement of the Poincar\'e Duality  stated in Theorem \ref{TT3} is derived from  the Poincar\'e Duality between Borel--Moore homology and cohomology of the open manifold $\tilde M.$   
\end{itemize} 

The paper ends up with a few topological applications, Observation \ref {C4} and Theorem \ref{T15}.
\vskip .1in
In section \ref{S2} one recalls the definition of  various spaces of configurations (of points with multiplicity, 
of submodules  and of pairs of submodules,  of mutually orthogonal closed Hilbert submodules of a
Hilbert module) 
 and of the relevant topologies on these spaces.  The configurations referred to above, $\delta ^f_r, \hat \delta ^f_r, \hat {\hat \delta} ^f_r,$ for $f: X\to \mathbb S^1$  an angle-valued map are defined 
 in section  \ref {S3} and all have the same support located
in $\mathbb T$ or $\mathbb C\setminus 0$.  
A point in $\mathbb T$ will be denoted by  $\langle a, b\rangle$ and  one in $\mathbb C\setminus 0$ by $z.$  

To formulate the results, for the reader's convenience  we recall some concepts and notations.\
\vskip .1in 

 For an angle-valued map  $f:X\to \mathbb S^1$ one denotes by 
 $\xi_f\in H^1(X;\mathbb Z)$ the integral cohomology class represented by $f$ and by 
 $\tilde f:\tilde X\to \mathbb R$  an infinite cyclic cover or {\it lift} of the map $f.$   
In section \ref{S2} one also defines the concepts of weakly tame and tame maps as well as of homologically regular and homologically critical values. The simplicial maps are always tame 
and then weakly tame. 
Informally, for a weakly tame map a {\it homologically regular} value is 
a complex number $z= e^{ia+b}\in \mathbb S^1$ s.t. the homology of the level of $z'$  in a small neighborhood of $z$ is unchanged and a {homologically critical} value
is a complex number $z= e^{ia+b}\in \mathbb S^1$ which is not homologically regular value.
  
For $\xi\in H^1(X;\mathbb Z)$ one denotes by:

\hskip.2in -  $C_\xi(X, \mathbb S^1),$ the space of continuous maps in the homotopy class defined by $\xi$ equipped with the compact open topology,  

\hskip .2in -  $\pi:\tilde X \to X,$ an infinite cyclic
cover defined by $\xi_f,$ or by $f.$

For a specified field $\kappa$ one denotes by:

\hskip .2in - $H^N_r(X; \xi),$ the Novikov homology  in dimension $r$ with respect to the field  $\kappa,$  

\hskip .2in - $\beta^N_r(X;\xi),$  the $r-$th Novikov--Betti number, cf section \ref{S2} for definitions.

In case $\kappa=\mathbb C$ the $L_2-$homology of $\tilde X$ in dimension $r$ will be denoted by $H_r^{L_2}(\tilde X).$ This is an $L^\infty(\mathbb S^1)-$Hilbert module.  In this case 
the von-Neumann dimension  of $H_r^{L_2}(\tilde X)$ is 
equals the Novikov--Betti number w.r. to the field $\mathbb C.$ 

The main results of this paper are collected in the following theorems: 
   
\begin{theorem} (Topological results) \label {TT1}\ 

\begin{enumerate}
\item   
The configurations $\delta^f_r, \hat \delta^f_r, \hat{\hat \delta}^f_r$ have the same support.  If $f$ is weakly tame and $\delta_r^f(z)
\ne 0,$  $z= e^{ia +(b-a)},$ then both $a$ and  $b$   
are homological critical values of $\tilde f.$ 
\item 
\begin{enumerate}
\item  $\sum_{z\in \mathbb C\setminus 0} \delta^f_r(z)=\beta^N_r(X;\xi_f),$ 
\item  $\bigoplus_{z\in \mathbb C\setminus 0} \hat \delta^f_r(z)\simeq H^N_r(X;\xi_f),$ 
\item When $\kappa= \mathbb C$ a $\mathbb C[t^{-1}, t]-$compatible inner product  on $H^N_r(X;\xi_f)$ (cf section \ref{S2} for definition) converts $\hat\delta ^f_r$ into a configuration 
$\hat{\hat\delta}^f_r$ of closed Hilbert submodules of $H^{L_2}_r(\tilde X)$ which satisfy 
\newline $\sum_{z\in \mathbb C\setminus 0} \hat{\hat \delta}^f_r(z)= H^{L_2}_r(\tilde X),$ and  $\hat{\hat \delta}^f_r(z)\perp \hat{\hat \delta}^f_r(z')$ for $z\ne z'.$
\end{enumerate}
\item If $X$ is a good ANR (cf section \ref{S2} for definition), in particular homeomorphic to a finite simplicial complex or to a compact Hilbert cube manifold,  then for an open and dense 
set of maps $f$ in $C_\xi(X,\mathbb S^1) $ one has $\delta^f_r(z)= 0\ \rm {or} \ 1.$
\end{enumerate}
\end{theorem}
Here and below "$=$``  denotes equality or canonical isomorphism and "$\simeq $ indicates  the existence of an isomorphism.

 Items 1. and 2.a) were first established in \cite {BH} but only for tame maps and by different methods. 

Anticipating section \ref{S2} we denote by  $\mathcal Conf_k(X)$ the set of configurations of $k$ points in $X$  and by $\mathcal CONF_V(X)$  the set of configurations of {\it subspaces} 
of the module $V$ indexed by the points of $X.$  If $V$ is a $\kappa-$vector space a {\it  subspace} means a genuine vector subspace, if $V$ is a free module a {\it subspace} means a 
free split submodule and if $V$ is a Hilbert module a {\it subspace} means a closed Hilbert submodule. 
In view of this notation 
item 2.a) indicates that $\delta^f_r\in \mathcal Conf_{\beta^N_r(X;\xi_f)} (\mathbb T),$ $T= \mathbb C\setminus 0),$
which can be identified to the ${\beta^N_r(X;\xi_f)}$-fold symmetric  product of $\mathbb C\setminus 0,$   
 hence to  the space of degree ${\beta^N_r(X;\xi_f)}$-monic polynomials with nonzero free  coefficient,    
hence to $\mathbb C^{{\beta^N_r(X;\xi_f)}-1}\times (\mathbb C\setminus 0)$ 
 and Item 2.b) implies that any family of splittings as defined in section {\ref{S3}}, makes $\hat \delta^f_r$  an element in $\mathcal CONF_V (\mathbb T),$ with $V$ the free  
$\kappa[t^{-1},t]-$module $ {H^N_r(X;\xi_f)}.$ 
Item 2.c) states  that $\hat {\hat \delta}^f_r$ is an element in $\mathcal CONF^O_{H^{L_2}_r(\tilde X)} (\mathbb T),$ the space configurations of mutually orthogonal closed 
Hilbert submodules of the $L^{\infty}(\mathbb S^1)-$Hilbert
 $H^{L_2}_r(\tilde X).$ 
 
Associated to $\xi$ there is the infinite cyclic cover $\pi:\tilde X\to X,$ 
a  principal $\mathbb Z-$ covering unique up to isomorphism, such that any continuous $f:X\to\mathbb S^1$ has lifts $\tilde f: \tilde X\to \mathbb R$ (i.e. $\mathbb Z-$ equivariant maps which 
induce by passing to $X=\tilde X/ \mathbb Z,$ the map $f$ cf. section \ref{S2})   unique up to an additive constant of the form $2\pi k.$ For two lifts $\tilde f$ of $f$ and $\tilde g$ of $g$ denote 
by $D(\tilde f, \tilde g)= \sup _{\tilde x\in \tilde X} |\tilde f(\tilde x)- \tilde g(\tilde x)|$ and  denote by $D(f,g)$ the minimal of $D(\tilde f, \tilde g)$ over all possible lifts of $f$ and $g,$ cf. 
section \ref{S2}. $D(f,g)$ provides a metric on $C_\xi(X,\mathbb S^1).$

\vskip .1in 
\begin{theorem} (Stability)\label{TT2}
Suppose $X$ is a compact ANR and $\xi\in H^1(X;\mathbb Z).$  

1. The assignment 
$$ C_{\xi}(X, \mathbb S^1)\ni f \rightsquigarrow \delta^f_r= P_r^f(z)\in \mathbb C^{\beta_r^N(X;\xi)-1}\times (\mathbb C\setminus 0)$$
 is a  continuous  map.

Moreover with respect to the canonical metric $\underline D$ provided by the identification of  
the space of configurations with
the $\beta^N_r(X;\xi)-$fold  
symmetric product of $\mathbb T,$  
one has the estimate $$\underline D (\delta^f , \delta^g) < 2 D(f,g).$$  

2. If $\kappa= \mathbb C$ and the spaces of configurations $\mathcal CONF_{H_r^{L_2}(\tilde X)}(\mathbb C\setminus 0)$ is equipped with either the {\it fine}  or the {\it natural} collision 
topology (cf section 2 for definitions) then the assignment $f \rightsquigarrow \hat{\hat \delta}^f_r$  is continuous. 
\end{theorem}
Item 1. was first  established in \cite{BH} for $X$ homeomorphic to a simplicial complex. 

\begin{theorem}  (Poincar\'e Duality) \label{TT3}
Suppose $M$ is a closed topological manifold of dimension $n$ \footnote {this theorem will be verified only in case $M$ is homeomorphic to a finite simplicial complex} which is $\kappa-$orientable 
and $f: M\to \mathbb S^1$  an angle-valued map with $\xi_f\ne 0.$ Then one has the following.
\begin{enumerate}
\item  
$\delta^f_r \langle a,b\rangle = \delta^{ f}_{n-r} \langle b, a\rangle,$ equivalently $\delta^f_r(z)= \delta^f_{n-r}(\tau z)$ with $\tau(z)= z |z|^{-2} e^{i\ln |z|}.$
\item The Poincar\'e Duality between Borel--Moore homology of $\tilde M$ and the cohomology of $\tilde M$ induces the isomorphisms of $\kappa[t^{-1},t]-$modules  
 from $H^N_r(M;\xi)$ to $H^N_{n-r}(M;\xi)$ which intertwine $\hat \delta^f_r \langle a, b\rangle $ and 
$\hat\delta^f_{n-r} \langle b, a\rangle.$  
Precisely a collection of compatible $N-$splittings $\mathcal S_r'$s  (cf definition \ref{D38} in section \ref{S3}), 
additional data which always exist,  provide the canonical isomorphisms of $\kappa[t^{-1},t]-$modules 
$$ PD^{S}_r: H^N_r(M;\xi_f)\to H^N_{n-r}(M;\xi_f)$$ 
$$ PD^{S}_r\langle a,b\rangle:\hat \delta^{\tilde f}_r\langle a,b\rangle   \to \hat \delta^{\tilde f}_{n-r}\langle b,a\rangle$$
$$ I^{S}_r:\oplus_{\langle a,b\rangle \in \mathbb T} \ \hat\delta_r\langle a, b\rangle \to H^N_r(M;\xi_f) \footnote {The finiteness of 
the rank of $H^N_r(M;\xi_f)$ implies that $\hat \delta^f_r\langle a, b\rangle= 0$ for all but finitely many pairs $\langle a, b\rangle$} $$
s.t. the diagram 
$$\xymatrix{ \oplus_{\langle a,b\rangle \in \mathbb T}\  \hat\delta_r\langle a, b\rangle\ar[d]_{I^S_r} \ar[rr]^{\oplus PD^{S}_{r}\langle a,b\rangle}&&
\oplus_{{\langle a,b\rangle \in \mathbb T}}\  \hat\delta_{n-r}\langle b, a\rangle \ar[d]_{I^S_{n-r}}\\
H^N_r(M;\xi_f)\ar[rr]^{PD^{S}_r} &&H^N_{n-r}(M;\xi_f)}$$
is commutative.
\item For $\kappa=\mathbb C$  a $C[t^{-1},t]-$compatible Hermitian inner product on
$H_r^N(M;\xi_f)$ (cf.section \ref{S2} for definition) provides  canonical compatible $N-$splittings $\mathcal S_r $   
such that  the von-Neumann completion (described in section \ref{S2}) leads to the canonical isomorphisms 
$$PD^{L_2}_r: H^{L_2}_r(\tilde M)\to H^{L_2}_{n-r}(\tilde M)$$ 
$$PD^{L_2}_r\langle a,b\rangle:\hat{\hat \delta}^{L_2}_r\langle a,b\rangle   \to \hat {\hat \delta}^{L_2}_{n-r}\langle b,a\rangle$$
$$I^{L_2}_r:\oplus_{\langle a,b\rangle \in \mathbb T}  \hat{\hat \delta}_r\langle a, b\rangle \to H^{L_2}_r(\tilde M)$$
which make the diagram
$$\xymatrix{ \oplus_{\langle a,b\rangle \in \mathbb T}\ \hat{\hat \delta}_r\langle a, b\rangle\ar[d]_{I^{L_2}_r} \ar[rr]^{\oplus PD^{L_2}_r\langle a,b\rangle}&&
\oplus_{\langle a,b\rangle \in \mathbb T} \ \hat{\hat\delta}_{n-r}\langle b, a\rangle \ar[d]_{I^{L_2}_{n-r}}\\
H^{L_2}_r(\tilde M)\ar[rr]^{PD^{L_2}_r} &&H^{L_2}_{n-r}(\tilde M)}$$
commutative.
A Riemannian metric  or a triangulation of a closed smooth or topological manifold  
 provides  canonical $\mathbb C[t^{-1},t]-$compatible Hermitian inner products on $H^N_r(M;\xi_f)$ and therefore the isomorphisms claimed in item 3. 
\end{enumerate}
\end{theorem} 
Item 1. was  first established in \cite{BH}.

Theorem \ref {TT3} item 2. implies the following

\begin {obs} \label {C4}\ 

If  $M$ is a compact manifold  with boundary $\partial M,$ 
and $H^N_r(\partial M; \xi_{f_{\partial M}})$ \footnote {with $f_{\partial M}$ the restriction of $f$ to $\partial M$}  
vanishes for all $r$ then $H^N_r(M;\xi_f)\simeq H^N_{n-r} (M;\xi_f).$
\end{obs}

If $X$ is connected and $u\in \overline \kappa\setminus 0,$ $\kappa$ the algebraic closure of $\kappa,$  denote by $(\xi,u)$  the local coefficient system defined by the rank one 
$\kappa-$representation 
 $$\xymatrix {(\xi,u): \pi_1(X,x_0)\ar[r] & H_1(X;\mathbb Z)\ar[r]^{\xi} &\mathbb Z\ar[r]^{u^ \cdot} & \overline \kappa\setminus 0}$$ 
where $\xi$ is interpreted as a group homomorphism $\xi: H_1(M.\mathbb Z)\to \mathbb Z$ and $u$ as the homomorphism $u^ \cdot(n)= u^n.$ Denote by $H_r(X;(\xi, u))$ the homology 
with coefficients in $(\xi,u)$ and by $\beta_r(X; (\xi,u))$ the dimension of this $\kappa-$vector space. 

For $\xi\in H^1(M;\mathbb Z)$ the set of Jordan cells $\mathcal J_r(X;\xi)$ was defined in \cite{BD11} and from a different perspective of relevance in the theorem below, discussed in \cite{BH} 
and \cite{B}. One denotes  by $\mathcal J_r(X;\xi)(u)$  the set of Jordan cells $(\lambda, k)$ with $\lambda = u\in \overline \kappa\setminus 0.$

\begin{theorem}\label {T15}
Suppose $(M^n,\partial M)$ is a compact manifold with boundary $\xi\in H^1(M;\mathbb Z)$ s.t.  $\beta^N_r(\partial M,\xi |_{\partial M})=0$ for all $r.$ Suppose that $M$ retracts by deformation 
to a simplicial complex of dimension $\leq [n/2],$ where $[n/2]$ denotes the integer part of $n/2$
and $\chi(M)$ is the Euler-Poincar\'e characteristic w.r. to the field $\kappa.$
Then we have.
\begin{enumerate}
\item  If $n=2k$ then one has:
\begin{equation}
\begin{aligned}
a)\  \beta^N_r(X;\xi)= &\begin{cases}  0 \ \rm{if}\  r\ne k\\ (-1)^k\chi (M) \rm{if}\  r = k\end{cases}\\
b)\  \beta_r(X) =&\begin{cases}  \mathcal J_{r-1}(X; \xi)(1) + \mathcal J_r(X; \xi)(1) \ \rm{if}\  r\ne k\\ \mathcal J_{k-1}(X; \xi)(1) + \mathcal J_k(X; \xi)(1)  +(-1)^k\chi (M), r=k \end{cases}, \\
c)\  \beta_r(X;\hat u \xi) =&\begin{cases}  \mathcal J_{r-1}(X; \xi)(u) + \mathcal J_r (X; \xi)(1/u)\ \rm{if}\  r\ne k\\  \mathcal J_{k-1}(X; \xi)(u) + \mathcal J_k(X; \xi)(1/u)  +(-1)^k\chi (M), r=k \end{cases}
\end{aligned}
\end{equation}

\item If $n=2k+1$ then and one has: 
\begin{equation}
\begin{aligned}
a)\  \beta^N_r(X;\xi)= \ &0 \\
b)\  \beta_r(X) =\ & 
 \mathcal J_{r-1}(X; \xi)(1) + \mathcal J_r(X; \xi)(1) \\
c)\  \beta_r(X;(\xi,u)) =\ &	  \mathcal J_{r-1}(X; \xi)(1/u) + \mathcal J_r(X; \xi)(u). 
\end{aligned}
\end{equation}
\item  If $V^{n-1}\subset M^{n}$ is compact  proper submanifold (i,e, $V\pitchfork \partial M, \footnote {$\pitchfork$= transversal}$ and $V\cap \partial M= \partial V$) whose Poincar\'e dual cohomology 
class is $\xi_f$ and $H_r(V)=0$ then the set of Jordan cells $J_r(M,\xi)$  is empty for $r>0$ and  $J_0(M,\xi)=\{(\lambda;1), \lambda=1\}$
\end{enumerate}
\end{theorem}

As pointed out to us by L Maxim, the complement $X = \mathbb C^n\setminus V$ of a complex hypersurface $V\subset \mathbb C^n, V:= \{(z_1, z_2, \cdots z_n) \mid f(z_1, z_2, \cdots z_n)=0\}$ regular at infinity, 
equipped with the canonical class $\xi_f\in H^1(X;\mathbb Z)$ defined by $f: X\to \mathbb C\setminus 0$ is an example of an open manifold 
which has as compactification a manifold with boundary equipped with a  degree one integral cohomology class which satisfies the hypotheses and then the conclusion of Theorem \ref{T15} above.  

Item 1. recovers    
 a calculation of L Maxim, cf \cite{M14} and \cite {FM16} \footnote {The  Friedl-Maxim results  state the vanishing of more general and more sophisticated $L_2-$homologies and Novikov type  
 homologies. Such results can be also recovered via the appropriate Poincar\'e Duality isomorphisms}   
that the complement of an algebraic hypersurface regular at infinity   
has vanishing Novikov homologies in all dimension but $n.$

At this point we thank L.Maxim for challenging questions  and informations about some of his work.    
I also thank the referee for providing repairs in the statement and the proof of Theorem 1.2, pointing out misprints and requesting useful additional explanations.

{\bf Acknowledgement: } This paper was written when the author was visiting MPIM-Bonn  (Nov. 2015-March 2016). He  thanks MPIM  for partial support during that period.  
\vskip .1in
\section{Preparatory material} \label {S2}

{\bf Angles and angle valued maps}
 
 An {\it angle} is a complex number $\theta= e^{it} \in \mathbb C, t\in \mathbb R$ and the set of all angles is denoted  by $\mathbb S^1=\{\theta= e^{it}\mid t\in \mathbb R\}.$  The space of 
 angles, $\mathbb S^1,$  is equipped with the distance 
 $$d(\theta_2, \theta_1)=\inf \{ |t_2- t_1|\mid
  e^{it_1}=\theta_1, e^{it_2}= \theta_2 \}.$$ One has $d(\theta_1, \theta_2)\leq \pi.$ With this description $\mathbb S^1$ is an oriented one dimensional manifold with the orientation provided 
  by a specified generator $u$ of $H_1(\mathbb S^1;\mathbb Z).$ 
  
A closed interval in $I\subset \mathbb S^1$ with ends the angles $\theta_1= e^{it_1}$ and $\theta_2= e^{it_2}$ is the set
$I:= \{e^{it} \mid  t_1\leq t\leq t_2, t_2-t_1 <2\pi \}.$
\vskip .2in

In this paper all real- or angle-valued maps  are  proper continuous maps defined on an ANR, hence locally compact  in case of real-valued  and compact  in case of angle-valued. Recall that an 
ANR, cf.\cite {Hu}, is a space homeomorphic to a closed subset $A$ of a metrizable space which has an open neighborhood $U$ which retracts to $A.$   Simplicial complexes, finite or infinite 
dimensional manifolds are ANR's.
\vskip .2in

\noindent {\bf Infinite cyclic cover}

For an angle-valued map  $f: X\to \mathbb S^1$ let 
$f^\ast: H^1(\mathbb S^1;\mathbb Z)\to H^1(X;\mathbb Z)$ be the   homomorphism induced by $f$ in integral cohomology and let  $\xi_f= f^\ast(u) \in H^1(X;\mathbb Z).$  
The assignment $f\rightsquigarrow \xi_f$ establishes a bijective correspondence between  the set of homotopy classes of continuous maps from $X$ to $\mathbb S^1$ and $H^1(X;\mathbb Z).$

Recall the following. 

 \hskip .1in   An infinite cyclic cover of $X$ 
is a map $\pi: \tilde X\to X$  together with  a free action $\mu:\mathbb Z\times \tilde X\to \tilde X$ such that  $\pi(\mu(n,x))= \pi(x)$ and the map induced by $\pi$ from $\tilde X/ \mathbb Z$ to $X$ 
is a homeomorphism. The infinite cyclic cover $\pi: \tilde X\to X$ is said to be  {\it associated to $\xi$} if any continuous proper map $\tilde f:\tilde X\to \mathbb R$ which satisfies $\tilde f(\mu(n,x))= \tilde f(x)+2\pi n$ 
induces a map  $f: X\to \mathbb R/2\pi\mathbb Z= \mathbb S^1$ with $\xi_f$ equal to $\xi.$ The homeomorphisms $\mu(k, \cdots): \tilde X\to \tilde X$ are called {\it deck transformations}.

 \hskip .1in  For two infinite cyclic covers  associated to $\xi,$  $\pi_i: \tilde X_i\to X, i=1,2$, there exists  homeomorphisms $\omega:\tilde X_1\to \tilde X_2$ which intertwines the free actions $\mu_1$ and $\mu_2$ 
 and  satisfy $\pi_2\cdot \omega= \pi_1.$  

 \hskip .1in Given $\pi: \tilde X\to X, \mu:\mathbb Z\times \tilde X\to \tilde X$ an infinite cyclic cover  
 and  $f:X\to \mathbb S^1$ an  angle  valued map 
 the map  
$\tilde f: \tilde X\to \mathbb R$ is a called a {\it lift of $f$}  if  $\tilde f(\mu(n,x))= \tilde f(x)+2\pi n$  and by passing to the quotients $X= \tilde X/\mathbb Z$  and $\mathbb S^1= \mathbb R/\mathbb Z$ the map $\tilde f$ 
induces exactly $f.$ A lift $\tilde f$ 
provides the following pull back diagram 
\begin{equation}
\xymatrix{ &\mathbb R\ar[r]^p &\mathbb S^1\\
&\tilde X \ar[r]^{\pi}\ar[u]_{\tilde f} &X\ar[u]_f.}
\end{equation} 
where  $p(t)$ is given by $p(t)= e^{it}\in \mathbb S^1.$  
Two lifts of $f$ , $\tilde f_1$ and $\tilde f_2$ differ by a deck transformation, i.e. there exists $k\in \mathbb Z$ with $\tilde f_2= \tilde f_1\cdot \mu(k,\cdots).$

Given $f:X\to \mathbb S^1$ there exists a canonical infinite cyclic cover $\pi: \tilde X\to X$ the pull back of  $p:\mathbb R\to \mathbb S^1,$ by $f$ and a canonical lift of $f,$ $\tilde f: \tilde X\to \mathbb R,$ 
 precisely  $\tilde X= \{x,r) \mid f(x)= p(t)\}$ and  $\tilde f(x,t)= t.$

Denote by $C_\xi(X;\mathbb S^1)$ the set of all continuous angle-valued maps on $X$ in the homotopy class defined by $\xi$  and let $\pi:\tilde X\to X$ be an infinite cyclic cover associated to $\xi.$
For $f, g \in  C_\xi (X, \mathbb S^1)$ and $\tilde h, \tilde l$ lifts of $f$ and $g$ 
denote by 
\begin{enumerate}
\item $D(\tilde h, \tilde l):= \sup_{\tilde x\in \tilde X} |\tilde h(x)- \tilde l(x)|$
\item $D(f,g):= \inf _{\begin{cases} \tilde h\  \text {lift  of} \  f \\ \tilde l \ \text {lift  of} \ g \end{cases}} D(\tilde h, \tilde l).$
\end{enumerate}
Note that if $d(f(x), g(x))<\pi$ then $D(f,g)= \sup _{x\in X} d(f(x), g(x))$ where $d$ is the standard metric on $\mathbb S^1,$ precisely $d(u,v) = \inf_{\theta, \theta'} | \theta-\theta'|$ for $u= e^{i\theta}, v= e^{i\theta'}.$  

\begin{obs}\label {O21} \ 

\begin {enumerate}
\item For any $f,g$ there exists lifts $\tilde f, \tilde g$ \ s.t. $D(f,g)= D(\tilde f, \tilde g).$
\item $D$ is a complete metric which induces the compact open topology on $C_{\xi}(X,\mathbb S^1).$
\item If $f, g\in C_{\xi}(X,\mathbb S^1),$ $\tilde f, \tilde g$ are liftings of $f$ and $g$  and $0=t_0< t_1 < t_2 <\cdots t_N <t_{N+1}=1$  any subdivision of the interval $[0,1],$ then the canonical homotopy $\tilde f_t= t \tilde f + (1-t) \tilde g$ satisfies 
$D(\tilde f, \tilde g)= \sum _{0\leq i\leq N} D(\tilde f_{t_i}, \tilde f_{t_{i+1}}).$ 
If $f_t$ 
denotes  the homotopy between $f$ and $g$ induced from $\tilde f_t,$ by passing to quotient  and $D(f,g)= D(\tilde f, \tilde g)$ then 
\begin{equation}D(f, g)\geq  \sum _{0\leq i\leq N}  D(f_{t_i}, f_{t_{i+1}}).\end{equation} 
\end{enumerate}
\end{obs}
The verifications are straightforward and left to the reader.

A homotopy $f_t$ as in item 3. is referred to as a {\it canonical homotopy}.

\vskip .2in
\noindent {\bf Regular and critical values and tameness} 

Let $f: X\to \mathbb S^1 \rm {or} \ \mathbb R$ be a proper continuous map.
\begin{itemize}

\item The value $s\in \mathbb S^1 \rm {or}\  \mathbb R$ is {\it regular / homologically regular}   if there exists a neighborhood $U$  of $s$ such that  
for any $s'\in U$ the inclusion   $f^{-1}(s')\subset f^{-1}(U)$ is a homotopy equivalence / homology  \footnote {w.r. to a specified field $\kappa$} equivalence and {\it critical / homologically critical } if not 
regular / homologically regular. We denote by $CR(f)$ the set of critical values and by $CRH^\kappa(f)$ the set of homologically critical values. Clearly $CRH^\kappa (f)\subseteq CR(f).$  Since all the time the field $\kappa$ will be  fixed once for all, $\kappa$ will be discarded from notation and we write $CRH(f)$  instead of $CRH^\kappa(f).$ 

\item The map $f$ is {\it weakly tame} if for any $s\in \mathbb S^1 \rm {or}\ \mathbb R$ the subspace $f^{-1}(s)$ is an ANR. This implies that for any closed interval $I$ in $\mathbb R$ or $\mathbb S^1$ $f^{-1}(I)$ is an ANR. 

\item  The map is {\it tame} if it is weakly tame and in addition the set of critical values is discrete and the distance between any two critical values bounded from below by a positive number $\epsilon(f)$. 

\item The map is {\it homologically tame} w.r. to a specified field if the set of homologically critical values is discrete and the distance between any two such homologically critical values is bounded from below by a positive number.

For an angle valued map $f:X\to \mathbb S^1$ consider $\tilde f: \tilde X\to \mathbb R$ a lift of $f.$
The map $f$ is weakly tame, resp. tame,  resp. homologically tame iff so is  $\tilde f.$  If $X$ is a finite simplicial complex then a map $f:X\to \mathbb S^1$ is called p.l. (piecewise linear) map if  some and then any lift $\tilde f :\tilde X\to \mathbb R$ is a p.l.map. 
An angle $\theta\in \mathbb S^1$ is a regular value,  resp. critical value, resp. homologically critical value if 
$\theta = e^{it}$ with $t$ a regular value, resp. critical value resp. homologically critical value for $\tilde f.$  
For technical reasons we will need the following concept.
\item A compact ANR $X$ is called a {\it good ANR} if the set of  
tame maps (real- or angle-valued maps) is dense in the set of all continuous maps w.r. to the compact open topology. In particular any finite simplicial complex is a good ANR  in view of the fact that the set of p.l.maps is dense in the set of continuous maps and each p.l. real-or angle valued map is tame.  
\end{itemize}
\vskip .2in

\noindent {\bf The von-Neumann completion:}

When $\kappa=\mathbb C,$ 
the ring of Laurent polynomials  $\mathbb \mathbb C[t^{-1}, t],$  equivalently the group ring $\mathbb C [\mathbb Z]$ of the infinite cyclic group $\mathbb Z,$  is an algebra with involution $\ast$ and trace $tr$ defined as follows. 
 
 If $a= \sum_{n\in \mathbb Z}  a_n t^n$ then:
\begin{equation*}
\begin{aligned}
 \ast (a):= a^\ast =& \sum_{n\in \mathbb Z} \overline a_n t^{-n} \\ tr (a)=&  a_0.
\end{aligned}
\end{equation*}
with $\overline a$ denoting the complex conjugate of the complex number $a.$

The algebra $\mathbb C[\mathbb Z]= \mathbb C[t^{-1},t]$ can be considered as a subalgebra of the algebra of bounded linear operators on the separable Hilbert space $$l_2(\mathbb Z) = \{a_n, n\in \mathbb Z \mid \sum_{n\in \mathbb Z}  |a_n|^2 <\infty\} $$  with the Hermitian scalar product $\mu(a,b)= \sum _{n\in \mathbb Z} a_n\overline b_n.$ 

The linear operator  defined by a Laurent polynomial (an alternative name for an element in $\mathbb C[\mathbb Z]= \mathbb C[t^{-1},t]$) is given by the multiplication of  the Laurent polynomial regarded as a sequence with all but finitely many components equal to zero with a sequences in $l_2(\mathbb Z).$   

One denotes by $\mathcal N$ the {\it weak closure} of $\mathbb C[\mathbb Z]$   in the space of bounded operators of the Hilbert space $l_2(\mathbb Z)$ when each  element of $\mathbb C[\mathbb Z]$ is regarded as such operator which is  a {\it finite von Neumann algebra}, with involution and trace extending the ones defined above, cf \cite{Lu97}. 

This algebra $\mathcal N$ is referred below as the von--Neumann completion of the group ring $\mathbb C[\mathbb Z]$ and is isomorphic to the familiar $L^\infty(\mathbb S^1)$ via Fourier series transform (whose inverse assigns to a complex valued function defined on $\mathbb S^1$ its Fourier series).   

Given a free $\mathbb C[t^{-1}, t]-$module  $M,$ a 
 {\it $C[t^{-1},t]-$compatible Hermitian inner product}  is a map $\mu: M\times M\to \mathbb C$ which satisfies:

\begin{enumerate}
\item linearity: $\mathbb C-$linear in the first variable,  
\item symmetry: $\mu(x,y) = \overline {\mu}(y,x),$   
\item positivity :  

a) $\mu(x,x)\in \mathbb R_{\geq 0}\subset\mathbb C$

b) $\mu(x,x)=0$ iff $x=0,$
\item for any $x,y$ there exists $n$ s.t.  $\mu(t^nx,y)=0$
\item $\mu(tx,ty)= \mu(x,y).$
\end{enumerate}
Items (1) to (4) make $\mu$ a non degenerate Hermitian inner product on $M$ and 
items (5) and (6) defines the $\mathbb C[t^{-1},t]-$compatibility of the Hermitian  inner product $\mu.$ 
\vskip .1in 
An equivalent data is provided by a $\mathbb C[t^{-1},t]-$valued inner product cf  \cite {Lu97} which is given by 
 a map $\hat \mu: M\times M\to \mathbb C[t^{-1},t]$ which satisfies: 
\begin{enumerate}
\item $\mathbb C[t^{-1},t]-$linear in the first variable,
\item symmetric in the sense that $\hat\mu(x,y)= \hat\mu(y,x)^\ast,$  $x,y\in M,$
\item positive definite  in the sense that it satisfies 
\begin{enumerate}
\item $\hat \mu(x,x)\in \mathbb C[t^{-1},t]_+$ with $\mathbb C [t^{-1},t]_+$ the set of elements of the form $a a^\ast$ and  
\item $\hat \mu(x,x)=0$ iff $x=0,$ 
\end{enumerate}
\item the map $\tilde{\hat \mu}: M\to {\it Hom}_{\mathbb C^[t^{-1},t]} (M,\mathbb C[t^{-1},t])$ defined by $\tilde{\hat \mu}(y)(x)= \hat \mu(x,y)$ is one to one. 
\end{enumerate}

The relation between $\mu$ and $\hat \mu$ is given by
\begin{equation}
\begin{aligned}
\hat\mu(x,y)= &\sum_{n\in \mathbb Z} t^n \mu(t^{-n}x,y)\\
\mu(x,y)= &tr  \hat \mu(x,y).
\end{aligned}
\end{equation} 

Clearly $\mathbb C[t^{-1}, t]$-valued  inner products exist. Indeed, if $e^1, e^2, \cdots e^k$ is a base of $M$ then $$\mu (\sum a_ie^i, \sum b_j e^j):= \sum a_i (b_i)^\ast$$ provides  such inner product.

Note that if $M$ is f.g. but not free,  a map "$\hat \mu$`` as above satisfying all properties but 3.(b) and instead satisfying  " $\ker \tilde {\hat\mu} $ equal to the  $\mathbb C[t^{-1},t]-$torsion of $M$``,  induces a {$C[t^{-1},t]-$compatible Hermitian inner product} $\mu$ on $M/ T M,$ where
$T M$ is the collection of torsion elements in $M.$  

By completing the  $\mathbb C-$vector space $M$ (the underlying vector space of the f.g. $\mathbb C[t, t^{-1}]-$module $M$) w.r. to the  Hermitian inner product $\mu$  one obtains a Hilbert space $\overline M$ which is an $\mathcal N$-Hilbert module, cf \cite {Lu97}, isometric to $l_2(\mathbb Z)^{\oplus k},$ $k$ the rank of $M.$

Two different $\mathbb C[t^{-1}, t]$-valued  inner products,  $\mu_1$ and $\mu_2,$ lead to the isomorphic (and then also isometric) Hilbert modules $\overline M_{\mu_1}$  and
$\overline M_{\mu_2}.$
This justifies discarding $\mu$  from notation. 

If one identifies $\mathcal N$ to $L^\infty (\mathbb S^1)$ and 
$l_2(\mathbb Z)^{\oplus k}$ to $L^2(\mathbb S^1)^{\oplus k}$  (by interpreting the sequence $\sum_{n\in \mathbb Z} a_n t^n$ as the complex valued function $\sum_{n\in \mathbb Z} a_n e^{i n\theta}$)  the $\mathcal N-$ module structure on $l_2(\mathbb Z)^{\oplus k}$  becomes the $L^\infty(\mathbb S^1)-$module structure on $(L^2(\mathbb S^1)^{\oplus k}$ 
given by the component-wise multiplication of $L^\infty-$functions with $L^2-$functions.

If $N\subset M$ is a free split submodule of the f.g free $\mathbb C[t^{-1},t]$-module $M$ and $\mu$ is an $\mathbb C[t^{-1}, t]$-valued inner product on $M,$ then $\overline N_{\mu}$ is a closed Hilbert submodule of $\overline M_{\mu}.$ Moreover if $N'_i\subseteq N_i \subseteq M,$ $i= 1,2,\cdots $ is a collection of pairs of split submodules then the collection $N_i/N'_i$ is a collection of free modules, quotient of submodules of $M,$ and the von Neumann completion process converts    
 $N_i'$ and $N_i$ into closed Hilbert submodules of $\overline M$ and each $N_i/N'_i$ into a Hilbert module canonically identified to the orthogonal complement of the kernel of the  projection $N_i\to N_i/ N_i'$ inside $N_i.$  The process of  passing from ($\mathbb C[t^{-1},t], M$) to  ($\mathcal N, \overline M$) referred above as  {\it von Neumann completion} was considered in \cite{Lu 95} for any group ring $\mathbb C[\Gamma] $ and f.g. projective $\mathbb C[\Gamma]$-module.    
\vskip .2in 

\noindent {\bf Configurations}
\vskip .1in
{\it (a) Configurations of points with multiplicity}

\noindent A configuration of points with multiplicity in $X$ is a map with finite support $\delta: X\to \mathbb Z_\geq 0.$  The support of $\delta$  is the set 
$$\supp\  \delta:= \{x\in X \mid \delta(x)\ne 0\}$$ and the cardinality of the support is 
$$\sharp \delta := \sum \delta(x).$$
Denote by $Conf_n(X)$ the set of configurations of cardinality $n.$  Clearly
$Conf_n(X)= S^n(X)= X^n/\Sigma_n$ the quotient of the $n-$fold  product $X^n$ by the action of the permutation group of $n-$elements, $\Sigma_n,$ and this description equips $Conf_n(X)$ with the quotient topology induced  from the topology of the product space $X^n.$ There is an alternative but equivalent way, see below, to describe this topology as {\it collision topology}.
\vskip .1in
{\it (b)Configuration of subspaces}

\noindent Let $A$ be a commutative ring with unit, in particular a field $\kappa= A,$ and $V$ a free module of finite rank $rank V= n$ and let 
$\mathcal S(V)$ be the set of split submodules of $V.$ 

A configuration of subspaces of $V$ indexed by points in $X$ is a map with finite support $\hat \delta:X\to \mathcal S(V)$ such that
$$\oplus i_x: \oplus \hat \delta(x)\to V,$$ with $i_x:\hat\delta(x)\to V$ the inclusion, is an isomorphism.
As before 
$$\supp\ \delta:= \{x\in X \mid \hat \delta(x)\ne 0\}.$$  

Denote by $\mathcal CONF_V(X)$ the set of configurations of such submodules (subspaces  if $V$ is a vector space). 
The configuration  $\hat\delta $ is called a {\it refinement of $\delta\in  Conf_{\rank V}(X)$}  if $\delta(x)= \rank \hat\delta(x).$

 If $\mathcal S(V)$ is equipped with a topology then $\mathcal CONF_V(X)$ carry a topology, the {\it collision topology}, defined by specifying for each element  a system of {\it fundamental neighborhoods}  \footnote {described in \cite{B1} in case  $A$ is a field}. 
  
  A fundamental neighborhood  of a configuration $ \hat \omega\in \mathcal CONF_V(X)$ with support $\{x_1, x_2,\cdots x_k\}$ and values $\hat \omega (x_i)= V_i$ is specified by: 
  
(a): a collection of disjoint open sets of $X,$  $(U_1, U_2, \cdots U_k),$ each $U_i$ neighborhood of $x_i,$  

(b): a collection of neighborhoods of $\hat\delta(x_i)$ in $\mathcal S(V),$  $\mathcal V_1\ni \hat\delta(x_1), 
\mathcal V_2\ni \hat\delta(x_2),\dots, \mathcal V_k\ni \hat\delta(x_k)$   

\noindent and consists of 
$$\{  \hat \delta\in  CONF_V(X) \mid  \sum_{x\in U_i\cap \supp \hat \delta} \hat \delta(x) \in \mathcal V_i\}.$$

If the topology on $\mathcal S(V)$ is the discrete topology then the topology on $CONF_V(V)$ is referred to as the {\it fine collision topology}. 

In case of configuration of points with multiplicity the  topology on  $\mathcal Conf_{\rank V} (X)$ can be described in the same way, simply by replacing $\mathcal S(V)$ by $\mathbb Z_{\geq 0}$ equipped with the discrete topology.
Note that the assignment 
$$CONF_V(X) \ni \hat \delta \leadsto \delta\in \mathcal Conf_{\rank V} (X)$$ is continuous. 

If $A=\kappa$ with $\kappa= \mathbb R $ or $\kappa=\mathbb C $  and the vector space $V$ (not necessary of finite dimension)  is equipped with a Hilbert space structure and $\mathcal S(V)$ is the set of closed subspaces, then one can consider on $\mathcal S(V)$ the topology induced from the {\it norm topology} on the space of bounded operators on $V.$
The closed subspaces of $V$ are identified to the self-adjoint projectors. In this case the corresponding topology  on $\mathcal CONF_V(X)$ is called the  {\it natural collision topology}. If $V$ is a Hilbert space the subset of configurations with the additional property that $\hat\delta(x)\perp \hat\delta(y)$ is denoted by $\mathcal CONF^O_V(X).$

\vskip .1in
{\it (c) Configurations of pairs}

\noindent Let $A$ be a commutative ring with unit, $V$ a free module of finite rank,  $rank V= N,$ and let 
$\mathcal P(V)$ be the set of pairs $(W, W'),$  $W\supseteq W'$ split submodules of $V.$ 
The pair $(W, W')$ is called called virtually trivial if $W=W'$

A {\it configuration of pairs} of submodules of $V$ parametrized by $X$ is a map $\tilde \delta: X\to \mathcal P(V)$ with  finite support 
$$\supp\  \tilde \delta:= \{x\in X\mid \tilde \delta(x) \ne \rm {virtually\  trivial}\}$$
which satisfies the following properties.

\begin{enumerate}
\item The set $\mathcal A= \tilde \delta(X)$ is finite,
\item  If $\alpha, \beta\in \mathcal A,$ $\alpha= (W_\alpha, W'_\alpha),$  $\beta= (W_\beta, W'_\beta)$
then $W_\alpha\subset W_\beta$ implies $W_\alpha\subseteq W'_\beta,$  
\item For any $\alpha$ one has 
 $$\sum _{\beta\in \mathcal A,  W_\beta\subseteq W_\alpha} \rank (W_\beta/ W'_{\beta})= \rank W_\alpha$$
and 
 $$\sum _{\alpha\in \mathcal A} \rank (W_\alpha/ W'_{\alpha})= \rank V.$$
\end{enumerate}
Any collection of splitting $\{s_\alpha: W_\alpha/ W'_\alpha \to W_\alpha \subseteq V, \alpha\in \mathcal A\}$, i.e. right inverses of the canonical projections $p_\alpha: W_\alpha\to W_\alpha/ W'_\alpha,$  assigns to $\tilde \delta$ the 
configuration of subspaces $\hat \delta$ defined by $\hat\delta(x)=s_{\tilde \delta(x)}(W_{\tilde \delta(x)}/ W'_{\tilde\delta(x)}).$ 
If $\kappa= \mathbb R \ \rm {or}\  \mathbb C$ and $V$ is a $\kappa-$Hilbert space  then the orthogonal complements provides canonical splittings  and the associated configuration $\hat \delta$
becomes a configuration of subspaces.  

If $A=\mathbb C[t^{-1},t]$ and $V$ is a free $A-$module of finite rank equipped with a $\mathbb C[t^{-1},t]$-valued inner product then the von-Neumann completion converts $A$ into $L^\infty(\mathbb S^1),$ $V$ into a finite type Hilbert module, hence a Hilbert space, and any configuration of pairs $\tilde \delta,$ by the process von Neumann completion, into a configuration of  Hilbert submodules.  First one converts $\tilde \delta$ into a configuration of pairs of Hilbert submodules and then, using the Hermitian inner product one realizes the quotient of each pair as a closed Hilbert submodule,   
Clearly the space of configurations of Hilbert submodules comes equipped with the natural collision topology as well as the fine collision topology. 
\vskip .2in

\noindent {\bf Novikov homology}

Let $\kappa$ be a field and let $\kappa[t^{-1},t]$  be the $\kappa-$algebra  of Laurent polynomials with coefficients in $\kappa.$ This is a commutative algebra which is an integral domain and a principal ideal domain. For a pair $(X, \xi)$ $\xi \in H^1(X;\mathbb Z),$ $X$ a compact ANR, let $\tilde X$ be the associated infinite cyclic cover and 
let $\tau :\tilde X\to \tilde X$ be  the positive generator of the group of deck transformation isomorphic to $\mathbb Z$ viewed as a homeomorphism of $\tilde X.$ Since $X$ is compact, 
the $\kappa-$vector space $H_k(\tilde X)$ is actually a finitely generated $\kappa[t^{-1},t]-$module whose multiplication by $t$ is given by the linear isomorphism induced by the homeomorphism $\tau.$  
 
 Since $\kappa[t^{-1},t]$ is a principal ideal domain   
 the collection of torsion elements  form a $\kappa[t^{-1},t]-$submodule $V_r(X;\xi):= \ \rm {Torsion} \ (H_r(\tilde X)) =TH_r(\tilde M)$  (usually referred to as {\it monodromy}) which, as 
 a $\kappa-$vector space, is of finite dimension. 
The quotient  module $H_r(\tilde X)/T H_r(\tilde X)$ 
is a finitely generated free $\kappa[t^{-1},t]-$module. In this paper this free $\kappa[t^{-1},t]-$module and its rank are  called the {\it Novikov homology} and the {\it Novikov--Betti number} and are denoted by $H^N_r(X; \xi)$ and $\beta^N_r(X; \xi)$ respectively  
\footnote  {Classically, the Novikov homology is the $\kappa[t^{-1},t]]-$vector space 
$H_r(\tilde X)\otimes_{\kappa[t^{-1},t]} \kappa[t^{-1},t]]$ with $\kappa[t^{-1},t]]$ the field of Laurent power series; clearly 
$\beta^N_r= \dim (H^N_r(\tilde X)\otimes_{\kappa[t^{-1},t]}\kappa[t^{-1},t]])= \rank (H^N_r(\tilde X)).$}. Since $\kappa[t^{-1}, t]$ is a principal ideal domain 
one has $H_r(\tilde X)\simeq H^N_r(X;\xi) \oplus T H_r(\tilde X) .$   

As pointed out above the $\kappa[t^{-1},t]-$module $V_r(X; \xi)= TH_r(\tilde X),$ which is finitely generated,  when regarded as a vector space over $\kappa$ is of finite dimension and the multiplication by $t$ is actually a $\kappa-$linear isomorphism $T.$  In view of the Jordan decomposition theorem it is completely determined up to isomorphism by  
the collection of pairs with multiplicity $$\mathcal J_r(X:\xi):=\{(\lambda, k) \mid \lambda\in \overline \kappa\setminus 0, k\in \mathbb Z_{\geq 1}\}.$$  Here $\overline \kappa$ denotes the algebraic closure of $\kappa.$ 
Recall that any such pair should be interpreted as a $k\times k$ matrix $T(\lambda, k)$ with $\lambda$ on diagonal, $1$ above diagonal and zero anywhere else  and by Jordan decomposition theorem $T$ (when regarded over $\overline \kappa$) is  similar to the direct sum of all these matrices $T(\lambda, k).$

For a fix $u\in \kappa\setminus 0$ one write $\mathcal J(X;\xi)(u)=\{(\lambda,k)\in \mathcal J_r(X;\xi) \mid  \lambda=u\}.$
\vskip .2in

\noindent {\bf $\kappa[t^{-1},t]-$modules}

A $\kappa[t,t^{-1}]-$module $V$ is actually a $\kappa-$vector space $V$ equipped with a $\kappa-$linear isomorphism $T:V\to V$. The multiplication by $t$ and the isomorphism $T$ are related by the formula $tv:= T(v)$. With this observation we define 
$V^\ast$ the $\kappa[t^{-1},t]-$module whose underlying vector space is the dual of $V,$ $Hom (V,\kappa),$ and linear isomorphism $T^\ast$ the dual of $T.$ 
Note that if $V$ is finitely generated $\kappa [t^{-1},t]-$module then $V^\ast$  is not finitely generated in general but only when is a torsion module.

If $\mathbb Z$ acts freely on the set $S,$ $\kappa [S]$ denotes the vector space  of $\kappa-$valued maps with finite support and $\kappa[[S]]$ denotes the vector space of all $\kappa-$valued maps then:
 
 \begin {enumerate}
 \item  Both $\kappa [S]$ and $\kappa [[S]]$ are $\kappa[t^{-1},t]-$torsion free modules, with  $\kappa[t^{-1},t]$-structure  induced by the action of $1\in \mathbb Z$ on $S.$
 \item $\kappa[[S]]$ is isomorphic to $\kappa[S]^\ast$  (as $\kappa[t^{-1},t]-$modules).
\item A torsion free $\kappa[t^{-1},t]$-module is finitely generated iff is isomorphic to $\kappa[S]$ for some free $\mathbb Z-$action on some set $S$ with the quotient set $S/\mathbb Z$  finite. 
\item   
If $V$ is a finitely generated torsion $\kappa [t^{-1},t]-$module then $V^\ast$  is a finitely generated torsion module and is isomorphic to $V$.
\end{enumerate}
\vskip .2in

\section {The configurations, $\delta^f_r, \tilde \delta^f_r, \hat{\hat \delta}^f_r$} \label {S3}  

\subsection {Boxes and the maps }
 
Let $\kappa$ be a fixed field. Consider
 $h: Y\to \mathbb R$  a (continuous) proper  map with $Y$ an ANR,  hence locally compact.  
For $a,b \in \mathbb R$  consider 

\begin{equation}
\begin{aligned}
\mathbb I^h_a(r)= &\img (H_r(h^{-1}(-\infty,a])\to H_r(Y)),\\
\mathbb I_h^b(r)= &\img (H_r(h^{-1}([b,\infty))\to H_r(Y)),\\
\mathbb F^{h}_r(a,b)= &\mathbb I^{h}_a(r)\cap \mathbb I^b_{h}(r) \subseteq H_r(Y).
\end{aligned}
\end{equation}

and denote by 
\begin{equation}
\begin{aligned}
\mathbb I^h_{-\infty}(r)= &\cap _{a\in \mathbb R}I^h_a(r)\\ 
\mathbb I_h^{\infty}(r)=& \cap _{b\in \mathbb R}I_h^b(r)\\
\mathbb F^{h}_r(-\infty,b)= &\mathbb I^{h}_{-\infty}(r)\cap \mathbb I^b_{h}(r) \subseteq H_r(Y)\\
\mathbb F^{h}_r(a,\infty)= &\mathbb I^{h}_a(r)\cap \mathbb I^{\infty}_{h}(r) \subseteq H_r(Y)\\
\mathbb F^{h}_r(-\infty,\infty)= &\mathbb I^{h}_{-\infty}(r)\cap \mathbb I^{\infty}_{h}(r) \subseteq H_r(Y)\\
\end{aligned}
\end{equation}
\begin{proposition}\label {P1} \ 
For $-\infty \leq a' <a, \  b<b'\leq \infty$ one has 
\begin{enumerate}
\item $\mathbb F^{h}_r(a',b')\subseteq \mathbb F^{h}_r(a,b),$     
\item $\mathbb F^{h}_r(a',b')= \mathbb F^{h}_r(a',b)\cap \mathbb F^{h}_r(a,b'),$
\item $\mathbb F^{h}_r(a,b)$  is a finite dimensional vector space. 
\end{enumerate}
\end{proposition}
\begin{proof}
Items 1. and 2. follows from definitions.
To check item 3. observe that by item 1. it suffices to verify the statement for $a\geq b.$ If $f$ is weakly tame the statement follows from the Mayer-Vietoris long exact sequence in homology in view of the finite dimensionality of $H_r( f^{-1} [b,a]),$  a consequence of the fact that $f^{-1} [b,a]$ is a compact 
ANR. If $f$ is only proper map one proceeds as in the proof of Proposition 3.2 in \cite {B1}.  Precisely one  dominates $X$ by a locally compact simplicial  complex 
$K$ and up to a proper homotopy the map $f$ by a simplicial proper map $g:K\to \mathbb R$ which is  weakly tame.  The result is true for $g$ by a simple Mayer-Vietoris argument and then and  is true for $f.$ 
\end{proof}
\vskip .1in
A  subset $B$ of $\mathbb R^2$ of the form $$B=(a',a]\times [b,b')$$  $-\infty\leq  a'<a, b<b'\leq \infty$  is called a {\it box}. When both $a',b'$ are finite the box is called {\it finite box}.
Denote by $$B(a,b,\epsilon):= (a-\epsilon,a]\times [b, b+\epsilon),$$ $0<\epsilon \leq \infty.$ 

Below we write $$B+c$$  for the box  which is {\it the $(c,c)-$translation along the diagonal of the box $B$}, precisely  $B+c:=(a'+c, a+c]\times [b+c, b'+c)$ and  for $B= (a',a]\times [b, b')$ 
we denote  by $cB\subset \mathbb R^2$ the set 
$$  cB:= B(a,b;\infty)\setminus B.$$ 

For a box $B= (a',a]\times [b,b')$ denote by 
\begin{equation}
\begin{aligned}
{\mathbb F'}^h_r(B):=&\mathbb F^{h}_r(a',b)+ \mathbb F^{h}_r(a,b')\subseteq \mathbb F^{h}_r(a,b)
\subseteq H_r(Y)\\
\mathbb F^h_r(B):=&\mathbb F^{h}_r(a,b)/ {\mathbb F'}^{h}_r(B).
\end{aligned}
\end{equation}

Clearly   if      $\mathbb I^h_{-\infty} (r)= \mathbb I^\infty_h(r) ,$ as it will be the case for $h$ a lift of a continuous angle valued map (cf. Proposition \ref{P36}), 
then for any $a,b\in \mathbb R$ 
\begin{equation*}
{\mathbb F'}^h_r(B(a,b; \infty)) = \mathbb I^h_{-\infty} (r) + \mathbb I_h^{\infty} (r)= \mathbb I^h_{-\infty}(r)= \mathbb I_h^\infty (r).
\end{equation*}

For $-\infty\leq a''<a'< a$ and $ b< b'< b''\leq \infty$ consider  
\begin{equation*}
\begin{aligned}
B'_1:=&(a',a]\times [b,b'),\\
B_1:= &(a'',a]\times [b,b'),\\
B''_1:=&(a'',a']\times [b,b'),
\end{aligned}
\end{equation*}
\begin{equation*}
\begin{aligned}
B'_2:=& (a',a]\times [b',b''),\\
B_2:= &(a'',a]\times [b',b''),\\
B''_2:=&(a'',a']\times [b',b'')
\end{aligned}
\end{equation*}
and 
\begin{equation}
B:=(a'',a]\times [b,b'')
\end{equation}
cf. Figure 1\ below.

\begin{tikzpicture} [scale=1.2]
\draw [<->]  (0,4) -- (0,0) -- (5,0);
\node at (-.2,3.5) {$b''$};
\node at (-0.2,1.5) {$b'$};
\node at (-0.2,1) {$b$};
\node at (1,-0.2) {$a''$};
\node at (3,-0.2) {$a'$};
\node at (4,-0.2) {$a$};
\draw [line width=0.10cm] (1,1) -- (4,1);
\draw [line width=0.10cm] (4,1) -- (4,3.5);
\draw [dashed, ultra thick] (1,1) -- (1,3.5);
\draw [dashed, ultra thick] (1,3.5) -- (4,3.5);
\draw [line width=0.10cm] (1,1.5) -- (4,1.5);
\draw [line width=0.10cm] (3,1) -- (3,3.5);
\node at (2,1.25) {$B''_1$};
\node at (2,2.5) {$B''_2$};
\node at (3.5,1.25) {$B'_1$};
\node at (3.5,2.5) {$B'_2$};
\node at(2.4, -1) {Figure 1};
\end{tikzpicture}
\hfill

One has
\begin{equation*}
\begin{aligned}
B_1= &B'_1 \sqcup B''_1\\
B_2= &B'_2 \sqcup B''_2\\
B\ = &B_1 \sqcup B_2 .\\
\end{aligned}
\end{equation*}
\vskip .1in 

\begin{proposition}\label {P2} \ 

The inclusions $B''_1\subseteq B_1 \supseteq B'_1,$  $B''_2\subseteq B_2 \supseteq B'_2$  and $B_1\subseteq B\supseteq B_2$ induce the short exact sequences 

$$\xymatrix {0\ar[r]&\mathbb F^{h}_r(B''_i)\ar[r]^{i^{B_i}_{B''_i,r}} & \mathbb F^{h}_r(B_i)\ar[r]^{\pi^{B'_i}_{B_i,r}}&\mathbb F^{h}_r(B'_i)\ar[r] &0},  i=1,2$$
$$\xymatrix {0\ar[r]&\mathbb F^{h}_r(B_2)\ar[r]^{i^{B}_{B_2,r}} & \mathbb F^{h}_r(B)\ar[r]^{\pi^{B_1}_{B,r}}&\mathbb F^{h}_r(B_1)\ar[r] &0}.$$ 
\end{proposition}
The proof follows from the definition of $\mathbb F^h_r(a,b)$ and Proposition \ref {P1} above.

\begin{obs} \label{O33}  If $B'$ and $B''$ are the boxes $B'= B_1'\sqcup B_2'$ and $B''= B_1''\sqcup B_2''$ then one has 

$$i^B_{B''_2,r}:= i^B_{B'',r}\cdot i^{B''}_{B''_2,r}= i^B_{B_2,r}\cdot i^{B_2}_{B''_2,r}$$
with $i^B_{B''_2,r}$ injective and 

$$\pi^{B'_1}_{B,r}:= \pi^{B'_1}_{B',r}\cdot \pi^{B'}_{B,r}= \pi^{B'_1}_{B_1,r}\cdot \pi^{B_1}_{B,r}$$
with $\pi^{B'_1}_{B,r}$ surjective. 
\end{obs}

\vskip .1in 
For $\epsilon'>\epsilon$ 
the inclusion $B(a,b;\epsilon)\subseteq B(a,b; \epsilon')$ for $\epsilon' >\epsilon$  induces the surjective linear map $$\pi^{B(a,b;\epsilon)}_{B(a,b;\epsilon'), r}: \mathbb F^{h}_r(B(a,b; \epsilon'))\to \mathbb F^{h}_r(B(a,b; \epsilon)).$$ 

Define $$\boxed{\hat \delta^{h}_r(a,b):= \varinjlim _{\epsilon\to 0}  \mathbb F^{h}_r((B(a,b; \epsilon))}.$$ In view of Proposition \ref{P1} $\hat \delta^{h}_r(a,b)$ is a finite dimensional vector space.  
 
Define  $$\boxed{\delta^{h}_r(a,b):= \dim \hat{\delta}^{h}_r(a,b)}.$$

Let ${\pi_{B,r}^{(a,b)} : \mathbb F^{h}_r(B)\to \hat\delta^{h}_r(a,b)}$ be given by 
$$\boxed{\pi_{B,r}^{(a,b)}:=\varinjlim _{\epsilon\to 0}  \pi^{B(a,b;\epsilon)}_{B, r}}.$$

\begin{proposition}\label{P3}\

\begin{enumerate}
\item For $a, b\in \mathbb R$ and $\epsilon$ is small enough ${\hat \delta}^{h}_r(a,b)= \mathbb F^{h}_r(B(a,b; \epsilon)).$ 
\item For any  box $B= (a',a]\times[b,b')$ with $-\infty\leq a'<a, b<b'\leq \infty$  the  set $\supp\ \delta^h_r \cap B$ has finite cardinality and one has
$$\sum _{(a,b) \in  B\ \cap\ \supp\ {\hat \delta} _r^h} {\delta}^{h}_r (a,b)= \dim \mathbb F^{h}_r(B).$$ 
\item If $h$ is weakly tame and
${\hat \delta}^{h}_r(a,b)\ne 0$  then both $a,b$ are homological critical values, hence 
\newline $\supp \ \hat \delta^{h}_r = \supp \  {\delta}^{h}_r  \subseteq CRH(h)\times CRH(h) \subseteq 
CR(h)\times CR(h).$ 
\end{enumerate}
\end{proposition}

\begin{proof} 
In view of finite dimensionality of $\dim \mathbb F_r^h(a,b)$ stated in Proposition \ref{P1} item 3. , for any $a< b$ there are at most finitely many values of $\alpha,$ $\alpha_1 < \alpha_2 <\cdots \alpha_k, \ a\leq \alpha_1\leq b$ s.t. $\dim \mathbb I^h_\alpha / \mathbb I^h_a$ or $\mathbb I^\alpha_h/ \mathbb I^b_h$ changes.  This implies also $\supp \delta^h_r \cap B$ has finite cardinality (hence the first part of item 2.).  The finite dimensionality  of $\dim \mathbb F_r^h(a,b)$ implies that $\dim \mathbb F^h(B(a,b;\epsilon))$ stabilizes when $\epsilon \to 0$ which implies item 1. 

To conclude  item 2. entirely, consider $a'=\alpha_0 <\alpha_1 <\cdots \alpha_r= a$ and $ b=\beta_1\cdots \beta_{s+1}= b'$ such that any box $B_{i,j}= (\alpha_{i-1}, \alpha_i]\times [\beta_j, \beta_{j+1})$ contains at most one point in $\supp \delta^h_r.$ Apply inductively Proposition \ref{P2} to derive that $$\sum _{\begin{aligned}1\leq i\leq r\\ 
1\leq j\leq s\end{aligned}} \dim \mathbb F^h(B_{i,j}) = \mathbb F^h_r(B).$$ 

 If $h$ is weakly tame, unless both $a,b$ are homologically critical values,  
$\mathbb F^h_r(B(a,b; \epsilon))$ stabilizes to zero which implies item 3.
Indeed in view of the definition one has $\dim \mathbb F_r(B(a,b; \epsilon))= \dim \mathbb F_r(a,b)+\dim \mathbb F_r(a-\epsilon,b-\epsilon)- \dim \mathbb F_r(a-\epsilon,b)-\dim \mathbb F_r(a, b-\epsilon).$ If either $a$ or $b$ are regular values and $\epsilon$ is small enough  the right side of the equality vanishes. 

\end{proof} 
\vskip .1in

In general $\delta^h_r$ and $\hat \delta^h_r$ are not configurations since their support, although discrete, might not be finite. 

\vskip .2in  

Consider the canonical surjective maps  

\begin{equation}
\begin{aligned}
\pi_r(a,b): \mathbb F^h_r(a,b)\to \hat \delta^{h}_r(a,b),\\
\pi^B_r(a,b): \mathbb F^h_r(a,b)\to \mathbb F^h_r(B), 
\end{aligned}
\end{equation}
Clearly $\pi_r(a,b) =  \pi ^{(a,b)}_{B,r} \cdot \pi^B_r(a,b)$
\vskip .1in 
One calls {\it splitting} any linear map 
\begin{equation}
\begin{aligned}
i_r(a,b): \hat\delta_r(a,b)\to \mathbb F^h_r(a,b), \ \rm {or}\\
i^B_r(a,b): \hat \delta_r(a,b)\to \mathbb F^h_r(B) 
\end{aligned}
\end{equation}
which satisfy 
\begin{equation*}
\begin{aligned}
\pi_r(a,b)\cdot i_r(a,b)=id, \  \rm {or}\\ 
\pi^{(a,b)}_{B,r} \cdot i^B_r(a,b)= id.
\end{aligned}
\end{equation*}
We continue to write $i_r(a,b)$ for its composition with the inclusion $\mathbb F^{h}_r(a,b)\subseteq H_r(Y).$ 
A splitting $i_r(a,b)$ provides the splitting $i^B_r(a,b)$ defined by $i^B_r(a,b)= \pi^B_r(a,b)\cdot i_r(a,b).$
\vskip .1in
For $(a,b)\in B'$ with $B'=(a',a_+]\times [b_-, b'')$ and $\infty\leq a' < a\leq a_+, \  b_-\leq b< b'\leq \infty,$ let  
$$ i^{B'}_r(a,b): \hat \delta ^f_r(a,b) \to \mathbb F^{h}_r(B')$$ be the composition 
$$\xymatrixcolsep{4pc}\xymatrix{ \hat \delta ^{h}_r(a,b)\ar[r]^{i_r(a,b)}&\mathbb F_r^{h}(a,b)\ar[r]^{\subseteq} &\mathbb F^{h}_r(a_+,b_-)\ar [r]^{\pi_r^{B'}(a',b')}&\mathbb F^{h}_r(B')}.$$ 

Both linear maps $i_r(a,b)$ and $i_r^{B'}(a,b) $ are injective. The first is injective because $\pi_r(a,b)\cdot i_r(a,b)= Id.$ The second is injective because of   
the commutativity of the diagram 
$$\xymatrix{
\mathbb F_r^{h}(a,b)\ar[d]^{\subseteq}\ar[r]^{\pi_r^{B}(a,b)} &\mathbb F^h_r(B)\ar[d]^{i^{B'}_{B,r}}\\
\mathbb F^{h}_r(a_+,b_-)\ar [r]^{\pi_r^{B'}(a_+,b_-)}&\mathbb F^{h}_r(B').}$$
which implies $i^{B'}_r(a,b)= i_{B,r}^{B'}\cdot i^B_r(a,b)$ with $i^{B'}_{B,r}(a,b)$injective by Observation \ref{O33} and $i^B_r(a,b)$ injective being a splitting.
\vskip .1in
One summarizes the above maps by the  diagram  below (\ref{E2}). 

\begin{equation}\label {E2}
\xymatrixcolsep{5pc}\xymatrix{ H_r( Y) 
& {\ \mathbb F^h_r(a,b)}\ar@{->>}[dd]_{\pi^B_r(a,b)}
\ar@{>->}[l]_\supseteq \ar@{->>}[r]_{\pi_r(a,b)}  
& {\hat \delta}^{h}_r(a,b)\ar@{-->}@/_1pc/[l] _{i_r(a,b)}\ar@{-->}[ddl]^{i ^B_r(a,b)} 
\ar@{>-->}[dd]^{i ^{B'}_r(a,b)}
\\
&&\\
&\mathbb F^h_r(B)\ar@{->>}@
/^1pc/[uur]^{\pi_{B,r}^{(a,b)}}\ar@{>-->}[r]^{i ^{B'}_{B,r}}
 &\mathbb F^h_r(B').}
\end{equation}
where the subdiagrams involving only arrows  $\rightarrow$ or only arrows $\dashrightarrow,$ are commutative and $i^B_r(a,b)= \pi^B_r(a,b)\cdot i_r(a,b).$

To simplify the writing, until the end of this section  we will write $\oplus _{(a,b)}$ resp. $\oplus _{(a,b)\in B}$ instead of $\oplus _{(a,b)\in \supp \ \delta^{h}_r}$ resp. 
$\oplus _{(a,b)\in \supp \ \delta^h_r \cap B}$  . 
\vskip .1in 
Choose a collection of splittings  $\mathcal S= \mathcal S_r=  \{ i_r(a,b)\mid (a,b)\in \supp\  \delta^{h}_r\}.$ 

Define $$\boxed{^\mathcal S \hat\delta^h_r(a,b):= i_r(a,b)(\hat \delta^h_r(a,b)) \subseteq\mathbb F_r^h(a,b)\subseteq H_r(Y)}$$ 
and consider  the map 
$$\boxed{ ^\mathcal S I_r= \bigoplus_{(a,b)} i_r(a,b): \bigoplus_{(a,b)} \hat\delta^{h}_r(a,b) \to H_r(Y),}$$  and for a box $B$   
the map  
$$\boxed{ ^\mathcal S I^B_r= \bigoplus_{(a,b)\in B} i^B_r(a,b): \bigoplus_{(a,b)\in  B} \hat\delta^{h}_r(a,b) \to \mathbb F^h_r(B)}$$  
with $i_r(a,b)$ and $i^B_r(a,b)$ provided by the splittings in the collection $\mathcal S.$

Denote by $$\boxed{\pi_r: H_r(Y)\to H_r(Y)/ (\mathbb I^h_{-\infty} +\mathbb I^{\infty}_h
)}$$ the canonical projection. 

\begin{proposition} \label{P34}\ 

Suppose $h$ is a weakly tame map.
\begin{enumerate}
\item The linear maps $^\mathcal S I^B_r$ and  $\pi_r\cdot ^\mathcal SI_r$ are isomorphisms. Therefore $^\mathcal SI_r$ is injective  and

 $^\mathcal S I_r(\oplus_{(\alpha, \beta)\in \mathbb R^2} \hat \delta^h_r(\alpha, \beta)) \cap (\mathbb I^h_{-\infty}(r) + \mathbb I_h^{\infty}(r))= 0$
\item 
\begin{enumerate}
\item $\sum_{\tiny \begin{cases} \alpha\leq a\\ 
\beta\geq b
\end{cases}} {^{\mathcal S}\hat\delta^h_r(\alpha, \beta)} + \mathbb I^h_{-\infty}(r) + \mathbb I_h^{\infty}(r)= \mathbb F^h_r(a,b)$
\item $\sum_{(\alpha, \beta)\in cB} {^{\mathcal S}\hat\delta^h_r(\alpha, \beta)} + \mathbb I^h_{-\infty}(r) + \mathbb I_h^{\infty}(r)= {\mathbb F'}^h_r(B)$
\item $\sum_{(\alpha, \beta)\in \mathbb R^2} {^{\mathcal S}\hat\delta^h_r(\alpha, \beta)} + \mathbb I^h_{-\infty}(r) + \mathbb I_h^{\infty}(r)= H_r(Y)$
\item $\sum _{(\alpha, \beta)\in B}  {^\mathcal S \hat\delta^h_r(\alpha, \beta)}  = {\mathbb F}^h_r(B).$
\end{enumerate}

\end{enumerate}
\end{proposition}

\begin{proof} \

Item 1.:  
To shorten the notation introduce the vector spaces
$\hat {\mathbb F}^{h}_r(B):= \oplus _{(a,b)\in B}  \hat \delta^{h}_r(a,b)$  and
$\hat {\mathbb F}^{h}_r:= \oplus _{(a,b)}  \hat \delta^{h}_r(a,b)$  and for the collection of splittings $\mathcal S$ 
one regards $^\mathcal SI^B_r$ and $^\mathcal SI_r$ as maps 
$$^\mathcal SI^B_r: \hat {\mathbb F}^{h}_r(B) \to  \mathbb F^{h}_r(B)$$
$$^\mathcal SI_r: \hat {\mathbb F}^{h}_r \to  H_r(Y). $$

For $B= B_1 \sqcup B_2$  with $B_1=B_1'',  B_2= B_1'$ or $B_1=B'', B_2= B'$  as in  Figure 1, 
one has the commutative  diagram 
$$\xymatrix{ \mathbb F^h_r(B_1) \ar[r] & \mathbb F^h_r(B) \ar[r] &\mathbb F^h_r(B_2)\\
\hat{\mathbb F}^h_r(B_1) \ar[u]_{^\mathcal S I^{B_1}_r}\ar[r]&\hat{\mathbb F}^h_r(B)\ar[u] ^{^\mathcal S I^{B}_r}\ar[r]&\hat{\mathbb F}^h_r(B_2) \ar[u]_{^\mathcal S I^{B_2}_r}.} 
$$
First one checks the statement (1.) for  boxes $B$ with $\supp \ \delta^{h}_r\cap B$ consisting of only one element. This is indeed  the case by Proposition \ref{P3} item  (3) for a box $B(a,b;\epsilon)$ with $\epsilon$ small enough. 

Manipulation with this diagram as in \cite {B1},  namely a decomposition of $B$ as a disjoint union of smaller boxes and successive applications of Proposition \ref {P3} permits to establish inductively the result for any finite box $B$. 
The general case and the isomorphism of $\pi_r\cdot ^\mathcal SI_r$ follows from  the case of $B$ a finite box by passing to projective limit as follows. 

Observe that because $\mathbb F^h_r(a,b)$ is finite dimensional, by Proposition \ref{P1} item 3. the cardinality of the set 
 $\supp\ \delta^{h}_r\cap B(a,b;R)$ remains constant when $R$ is large enough.

 Consider :  
 $\hat{\mathbb F}^h_r(B(a,b;\infty)):= \bigoplus_{(a,b)\in (B (a,b:\infty))} \  \hat \delta ^{h}_r(a,b).$\ \   
Since the set  "$\supp \delta^{h}_r\cap B (a,b;R)$ is constant when $R$ is large"  one has $\hat{\mathbb F}^h_r(B (a,b;\infty))= \varprojlim_{R\to \infty} \hat {\mathbb F}^h_r(B(a,b;R)).$

Consider
 $\mathbb F^h_r(B(a,b;\infty)): = {\mathbb F^h_r(a,b)}/( \mathbb I^h _{-\infty}(r)\cap \mathbb I^{b}_h(r) + \mathbb I^h_{a}(r)\cap \mathbb I^\infty_h(r)) .$ For the same reason  
 $\mathbb F^h_r(B(a,b; \infty))= \varprojlim_{R\to \infty} \mathbb F^h_r(B(a,b;R)).$ 
 
Since  $^\mathcal S I^{B(a,b;R)}_r$ is an isomorphism for any $R,$ $^\mathcal S I^{B(a,b;R)}_r$ stabilizes in $R$ 
and  $^\mathcal S I^{B(a,b;\infty)}: = \varprojlim _{R\to \infty} ^\mathcal SI^{B(a,b;R)}_r$  one has  $^\mathcal S I^{B(a,b;\infty)}$
is an isomorphism.

Note that $\mathbb R^2= \cup_L B(-L,L;\infty)$ and $^\mathcal S I_r ^{\mathbb R^2}= \varinjlim_{L\to -\infty} \  ^\mathcal S I_r^{B(-L,L;\infty)}.$ 
 Since  $^\mathcal S I_r^{B(-L,L;\infty)}$ is an isomorphism for any $L$ so is  $^\mathcal S I_r ^{\mathbb R^2}$
which is exactly  $\pi_r \cdot (^\mathcal S I_r).$

Item 2.: Proposition \ref{P3} and item 1. imply item 2.

\end{proof}
\vskip .1in 
An immediate consequence of  Proposition \ref {P34} is the following corollary. 

\begin {corollary} \label {C35}\
For a discrete collection of points $(a_i, b_i)\in \mathbb R^2, i\in \mathcal A$ 
\begin{enumerate}
\item $\cap_{i\in \mathcal A} \mathbb F^h_r(a_i, b_i) = \oplus _{\{(\alpha, \beta)\in \cap_{i\in \mathcal A} B(a_i, b_i; \infty)\}}      {^\mathcal S \hat \delta^h_r(\alpha, \beta))} +\mathbb I^h_{-\infty}(r) + \mathbb I_h^{\infty}(r),$
\item $\cup_{i\in \mathcal A} \mathbb F^h_r(a_i, b_i) = \oplus _{\{(\alpha, \beta)\in \cup_{i\in \mathcal A} B(a_i, b_i; \infty)\}}       {^\mathcal S\hat \delta^h_r(\alpha, \beta))} +\mathbb I^h_{-\infty}(r) + \mathbb I_h^{\infty}(r).$
\end{enumerate}
\end{corollary}

\vskip .2in
\subsection {Definition and properties of $\delta^f_r$ and $\hat \delta ^f_r.$} 

Let $f:X\to \mathbb S^1$ be a continuous map, $X$ compact ANR,   $\tilde f:\tilde X\to \mathbb R$ a lift of $f,$ and $\kappa$ a fixed field. 
We apply 
the previous considerations to $\tilde f: \tilde X\to \mathbb R.$  In this case we have the deck transformation $\tau: \tilde X\to \tilde X$ which induces the isomorphism $t_r: H_r(\tilde X)\to H_r(\tilde X)$ 
and therefore a structure of $\kappa[t^{-1},t]-$module on this $\kappa-$vector space. 
Recall that for a box $B=(a',a]\times [b,b')$  one denotes by $B+c$ the box 
$B+c:= (a'+c, a+c]\times [b+c, b'+c).$

\begin{obs} \label {O35}\

\begin{enumerate}
\item The isomorphism $t_r$ satisfies $t_r(\mathbb F^{\tilde f}_r(a,b))= \mathbb F^{\tilde f}_r(a+2\pi,b +2\pi)$ and $t_r^{-1}(\mathbb F^{h}_r(a,b))= \mathbb F^{h}_r(a-2\pi,b -2\pi).$ 
\item For any box $B= (a',a]\times [b,b')$ consider the box $B+2\pi.$
The isomorphism $t_r$ induces 
the isomorphisms $t_r(B): \mathbb F_r^{\tilde f}(B) \to \mathbb F_r^{\tilde f}(B+2\pi)$ and then $\hat t_r(a,b): \hat \delta^{\tilde f}_r(a,b) \to \hat \delta^{\tilde f}_r (a+2\pi , b+2\pi).$ 
\item $\mathbb I^{\tilde f}_{-\infty}(r)$ 
and $\mathbb I^\infty_{\tilde f}(r)$ are invariant w.r. to $t_r$ , 
hence $\kappa[t^{-1}, t]-$submodules, therefore 
$H_r(\tilde X)/ (\mathbb I ^{\tilde f}_{\infty}(r)+ \mathbb I_{\tilde f}^\infty (r)$ is a $\kappa [t^{-1}, t]-$ module. 
\end{enumerate}
\end{obs}

Clearly the diagram (\ref{E313}) with the vertical arrows $t_r(a,b), t_r(B), \hat t_r(a,b)$induced by $t_r$ is commutative. 

\begin{equation}\label {E313}
\xymatrix{
H_r(\tilde X)\ar[dd]^{t_r} &\mathbb F^{\tilde f}_r(a,b)\ar[l]_{\supseteq}\ar[rr]^{\pi_r(a,b)}\ar[dd]^{t_r(a,b)}\ar[rd]_{\pi^B_{ab,r}} &&\hat\delta^{\tilde f}_r(a,b)\ar[dd]^{\hat t_r(a,b)}\\
& &\mathbb F^{\tilde f}_r(B)\ar[ru]^ {\pi ^{(a,b)} _{B,r}}\ar[dd]^<<<<{tr(B)} &\\ 
H_r(\tilde X) &\mathbb F^{\tilde f}_r(a+2\pi,b+2\pi)\ar[l]_- {\supseteq}\ar[rd]^{\pi^{B'}_r(a+2\pi, b+2\pi)}\ar[rr]^{\pi_r(a+2\pi,b+2\pi)} &&\hat\delta^{\tilde f}_r(a+2\pi,b+2\pi)\\
& &\mathbb F^{\tilde f}_r(B+2\pi)\ar[ru]^{\pi ^{(a+2\pi,b+2\pi)}_{B+2\pi,r}} &}\end{equation}
\begin{proposition}\label {P36}\ 
$\mathbb I^{\tilde f}_{-\infty}(r) = \mathbb I^{\infty}_{\tilde f} (r)= T(H_r(\tilde X)).$
\end{proposition} 

\begin{proof}
If $x \in T(H_r(\tilde X)) $ then there exists  an integer $l\in \mathbb Z$ and  a polynomial $P(t)= \alpha_n t^n +\alpha_{n-1} t^{n-1} \cdots \alpha_1 t + \alpha_0,$ $\alpha_i\in \kappa,$ $\alpha_0\ne 0$ such that $P(t)\ t^l x=0.$  Let $y=t^l x.$  
Since $H_r(\tilde X)= \cup_b \mathbb I^b_{\tilde f}(r)$ one has $y\in \mathbb I^b(r)$ for some $b \in \mathbb R.$ 
Since  $P(t) y=0$  one concludes that  
 $$y= - (\alpha_n/\alpha_0) t^{n-1} \cdots - (\alpha_1/\alpha_0) t y$$ 
and therefore $y\in \mathbb I^{b+2\pi}(r).$  

Repeating the argument  one concludes that $y\in I^{b+2\pi l}$ for any $l,$ hence $y \in \mathbb I^\infty(r).$ 
Since $x= t^{-l} y,$  
one has  $x\in \mathbb I^\infty (r).$ Hence $T(H_r(\tilde X))\subseteq I^\infty (r).$

Let $x\in \mathbb I^\infty(r).$  Since $H_r(\tilde X)= \cup _a \mathbb I_a^{\tilde f}(r)$ then $x\in \mathbb I_a(r)$ for some $a\in \mathbb R,$ and if in addition $x\in \mathbb I^\infty(r)$ then  by Observation \ref{O35}  3.  all $x, t^{-1}x, t^{-2} x, \cdots t^{-l}x,\cdots \in \mathbb I_a(r)\cap \mathbb I^\infty(r).$  Since by Proposition  \ref{P1} (3.) 
 the dimension of $\mathbb I_a(r)\cap \mathbb I^\infty(r)$ is finite, there exists $\alpha_{i_1}, \cdots \alpha_{i_l}$ such that $$(\alpha_{i_1} t^{-i_1} + \cdots \alpha_{i_l} t^{-i_l}) x=0.$$  This makes $x\in T(H_r(\tilde X))$. Hence $\mathbb I^\infty(r)\subseteq T(H_r(\tilde X)).$  
Therefore $\mathbb I^\infty(r) = T(H_r(\tilde X)).$ 
 By a similar argument one concludes that $H_r(\tilde  X)= \mathbb I_{-\infty}(r).$

\end{proof}

\medskip

Recall that:
\begin{itemize}
\item 
$H^N_r(X;\xi_f):= H_r(\tilde X)/ T(H_r(\tilde X)),$  
 \item 
$\pi(r): H_r(\tilde X)\to H^N_r(X;\xi_f)$  denotes the canonical projection and 
\item   the $\kappa-$vector spaces,  $H_r(\tilde X), T(H_r(\tilde X)), H^N_r(X;\xi_f),$ are $\kappa[t^{-1},t]-$modules with the multiplication by $t$ given by or induced by the isomorphism $t_r$ with $\pi(r)$ is $\kappa[t^{-1},t]-$linear. 
\end{itemize}
In view of Proposition \ref {P36} \  $T(H_r(\tilde X))$ is contained in $\mathbb F^{\tilde f}_r(a,b)$ and  $\mathbb F'^{\tilde f}_r(B) $
\  and then one defines 
\begin{itemize}
\item $^N\mathbb F^{\tilde f}_r (a,b): = {\mathbb F}^{\tilde f}_r(a,b) / T(H_r(\tilde X)),$ 

for any $B:=(a',a]\times [b, b')$  
\item $^N\mathbb F^{' \tilde f}_r (B): = {\mathbb F}^{' \tilde f}_r(B) / T(H_r(\tilde X)),$  $^N\mathbb F^{\tilde f}_r (B): =  ^N\mathbb F^{\tilde f}_r (a,b) /   ^N\mathbb F^{' \tilde f}_r(B),$ 
and then  
\item$ ^N{\hat\delta}^{\tilde f}_r(a,b)= \lim _{\epsilon \to 0}   {^N\mathbb F^{\tilde f}_r (B(a,b;\epsilon)}=^N\mathbb F^{\tilde f}_r (B(a,b;\epsilon))$ for $\epsilon$ small enough.
\end{itemize}
Clearly one has:
\begin{itemize}
\item 
$^N\mathbb F^{' \tilde f}_r (B)\subseteq ^N\mathbb F^{\tilde f}_r (a,b)\subset H^N_r(X;\xi_f),$ 

\item 
\begin{enumerate}
 \item $ \mathbb F^{\tilde f}_r(B)= ^N\mathbb F^{\tilde f}_r(B)  $ 
  \item $\hat \delta^{\tilde f}_r(a,b)= ^N \hat \delta^{\tilde f}_r(a,b) ,$
 \end{enumerate}
\item the 
 diagram below (\ref{D10}) is commutative with the vertical arrows isomorphisms.
\end{itemize}
\begin{equation} \label {D10}
\xymatrix {
\hat\delta ^{\tilde f}_r(a,b)\ar[dd]^{\hat t_r(a,b)}& ^N \mathbb F^{\tilde f}_r (a,b)\ar[l]\ar[dd]^{^N\hat t_r(a,b)}\ar[rd]\ar[rr] &&H^N_r(X,\xi_f)\ar[dd]^{^N\hat t_r}
\\
&&{^N\mathbb F^{\tilde f}_r(B)=\mathbb F^{\tilde f}_r(B)}
\ar[llu]\ar[dd]^<<<<<{^N\hat t_r(B)}&\\
\hat \delta ^{\tilde f}_r(a+2\pi,b +2\pi) & ^N\mathbb F^{\tilde f}_r(a+2\pi,b+2\pi)\ar[l]\ar[rd]\ar[rr]&&H^N_rX,\xi_f)\\
&&{^N\mathbb F^{\tilde f}_r(B+2\pi)= \mathbb F^{\tilde f}_r(B+2\pi)\ar[llu]}
\ar[llu]&.}
\end{equation}

\vskip .2in
Recall from Introduction that  $\langle  \rangle : \mathbb R^2\to \mathbb T=\mathbb R^2/\mathbb Z$ denotes the map 
which assigns to $(a,b)\in \mathbb R^2$ its equivalence class $\langle a, b\rangle\in \mathbb T.$  
One denotes by $\langle K\rangle \subseteq \mathbb T$  the image of $K\subseteq \mathbb R^2$ by the map $\langle\rangle$  in particular one writes  
$\langle a, b\rangle, \ \langle B\rangle,\
\langle cB \rangle $ for the image of $(a,b), \ B, \ cB.$

The  box $B= (a-\alpha, a]\times [b, b+\beta)$   is called {\it small} if $0< \alpha, \beta <2\pi,$ in which case the restriction of $\langle \rangle$ to $B$ is one to one;  
clearly
if $B$ is a small box so is any $(B+c)$  and  $(B+2\pi k)\cap (B+2\pi k') =\emptyset$ for $k \ne k' .$  
\vskip .2in 

For $\langle a,b \rangle \in \mathbb T$  and $\langle B\rangle \subseteq \mathbb T$ with $B= (a-\alpha, a]\times[b, b+\beta)$ a small box introduce:
\begin{itemize}
\item
$^N \mathbb F^{f}_r\langle a,b \rangle:= \sum_{k\in \mathbb Z} ^N \mathbb F^{\tilde f}_r(a+2\pi k, b+2\pi k) \subseteq H^N_r(X; \xi_f),$
\item $^N \mathbb F'^ {f}_r\langle B\rangle:=  \sum_{k\in \mathbb Z} \  ^N\mathbb F'^{\tilde f}_r(B+2\pi k)= (^N\mathbb F^{\tilde f}_r\langle a',b\rangle + ^N\mathbb F^{\tilde f}_r \langle a,b'\rangle)  \subseteq ^N\mathbb F^{\tilde f}_r\langle a,b\rangle\subseteq H^N_r(X;\xi_f)$

both  $\kappa[t^{-1},t]-$submodules of the free $\kappa[t^{-1},t]-$module  $H_r^N(X;\xi_f),$ hence f. g. free modules,

\item $\mathbb F^f_r\langle B\rangle := \oplus_{k\in \mathbb Z}  \mathbb F^{\tilde f}_r(B+2\pi k)=  
 \oplus_{k\in \mathbb Z} \  ^N\mathbb F^{\tilde f}_r(B+2\pi k),$ 
\item $\hat \delta ^f_r\langle a, b\rangle:= \oplus _{k\in \mathbb Z}\hat \delta^{\tilde f}_r(a+2\pi k, b+2\pi k)  = \oplus _{k\in \mathbb Z}\ ^N\hat \delta^{\tilde f}_r(a+2\pi k, b+2\pi k)$  

 both ($\mathbb F^f_r\langle B\rangle $ and  $\hat \delta ^f_r\langle a, b\rangle$)  free $\kappa[t^{-1},t]-$modules 
whose multiplication by $t$ is given by  the isomorphism  $\oplus_{k\in \mathbb Z} \hat t_r(a+2\pi k, b+2\pi k).$  
\end{itemize}
\vskip .2in 

Recall that for a set $S$ equipped with an action $\mu:\mathbb Z\times S\to S$ the $\kappa-$vector space $\kappa[S],$ of $\kappa-$valued finitely supported maps, has a structure of $\kappa[t^{-1},t]-$module which is free when the action is free and has a base indexed by the quotient set $S/\mathbb Z.$ If $S\subset \mathbb R^2$ is a discrete subset,  invariant to the action $\mu:\mathbb Z\times \mathbb R^2\to \mathbb R^2,$ given by $\mu(n, (a,b))= (a+2\pi n, b+ 2\pi n),$  then $\kappa[S]$ is a free $\kappa[t^{-1},t]-$module with a base indexed by $\langle S\rangle.$ 
For the box $B=(a-\alpha, a]\times [b, b+\beta)\subseteq \mathbb R^2$ one denotes by  $\widehat B$ and $\widehat cB$  the subsets in $\mathbb R^2$ 
\begin{equation}
\begin{aligned} 
 \widehat B:=\cup _{z\in \mathbb Z} (B+2\pi k)  \subset \mathbb R^2,\\
 \widehat {cB}:= \cup_{k\in \mathbb Z}\  c(B+2\pi k) \subset \mathbb R^2
 \end{aligned}\end{equation} 
   and by $\langle \widehat B\rangle \subseteq \mathbb T$ and $\langle \widehat cB\rangle \subseteq \mathbb T$ their images by the map $\langle\rangle.$ 

Clearly the sets 
$\widehat B$, $\widehat cB$ and $\supp\ \delta^{\tilde f}_r$ are invariant to the free action $\mu$ with quotient sets $\langle \widehat B\rangle $, $\langle \widehat cB\rangle $ and 
$\langle \supp\ \delta^{\tilde f}_r\rangle= \supp\ \delta^{f}_r$ respectively.
In view of the above and of  Proposition \ref{P34}  one can conclude that  the $\kappa-$vector spaces 
$$^N \mathbb F^f_r\langle a,b\rangle,  ^N {\mathbb F'}^f_r\langle B\rangle ,^N \mathbb F^f_r\langle a,b\rangle / ^N {\mathbb F'}^f_r\langle B\rangle$$ are free $\
\kappa[t^{-1},t]-$modules with bases indexed by 
$\langle \widehat B(a,b;\infty)\rangle \cap \supp\ \delta^{f},$   $\langle \widehat cB\rangle \cap \supp\ \delta^{f}$  and $\langle \widehat B\rangle \cap \supp\ \delta^{f}$ 
the quotient sets of 
$ \widehat B(a,b;\infty)\cap \supp\ \delta_r^{\tilde f},$ \ \   $ \widehat cB \cap \supp\ \delta_r^{\tilde f}$  and $ \widehat B \cap \supp\ \delta_r^{\tilde f}$ respectively.

For any  box $B= (a-\alpha,a]\times [b,b+\beta)$ and  $k\in \mathbb Z$ consider the $\kappa-$linear map 
 
$$\iota_r(a,b;k) : \frac{
^N\mathbb F^{\tilde f}_r (a+2\pi k, b+2\pi k)}{ ^N\mathbb {F'}^{\tilde f}_{r} (B+2\pi k)} \to 
 ^N \mathbb F^{f}_r\langle a,b \rangle / ^N \mathbb F'^ {f}_r\langle B\rangle
$$ 
induced by the inclusion 
$^N\mathbb F^{\tilde f}_r (a+2\pi k, b+2\pi k)\subset  \oplus_{k'\in \mathbb Z} ^N\mathbb F^{\tilde f}_r (a+2\pi k', b+2\pi k')$
 and let 
 $$\iota_r\langle a, b\rangle:=\oplus _{k'\in \mathbb Z} { \iota_r(a,b;k')}: ^N\mathbb F^f_r\langle B\rangle  \to  ^N\mathbb F^f_r\langle a, b\rangle / ^N {\mathbb F' }^f\langle B\rangle.$$
This map is surjective and, in view of the commutative diagram (\ref {D10}), is $\kappa[t^{-1},t]-$linear.

Since $\iota_r (a,b)$ is surjective and both the source and the target  are free modules  of equal finite rank it follows that $\iota_r(a,b)$ is an isomorphism.

One summarizes the above observation as Proposition \ref {P37} below.

\begin{proposition} \label {P37}\ 

1. The $\kappa[t^{-1},t]-$module $^N \mathbb F^{f}_r\langle a,b \rangle / ^N \mathbb F'^ {f}_r\langle B\rangle$ is free and  of rank  $\sharp (\langle \widehat B\rangle \cap \supp \delta^f_r).$ \
If $B$ is a small box then this rank is equal to $\sharp (B\cap \supp \delta^{\tilde f}_r).$

2. If $B$ is small then $\kappa-$linear map  $\iota_r\langle a, b\rangle $  is an isomorphism. 

3. For $\epsilon$ small enough $\hat \delta^f \langle a, b\rangle = ^N\mathbb F_r\langle B(a,b; \epsilon)\rangle$  with $B(a,b; \epsilon)= (a-\epsilon]\times [b, b+\epsilon).$
\end{proposition}

\vskip .1in

Let $^N\pi_r(a,b): ^N\mathbb F^{\tilde f}_r(a,b)\to \hat \delta^{\tilde f}_r(a,b)$ be the canonical projection. 
As  in the definition of {\it splittings}  an $N-$splitting is a linear map $^N i_r(a,b):\hat \delta^{\tilde f}\to ^N\mathbb F^{\tilde f}_r(a,b)$  s.t. $^N\pi_r(a,b)\cdot ^N i_r(a,b)= id.$

\begin{definition}\label {D38}
A collection of splittings $\mathcal S:=\{  i_r(a,b):\hat \delta^{\tilde f}_r(a,b)\to \mathbb F^{\tilde f}_r(a,b), \ a,b \in \mathbb R\}$ resp. $N-$splittings $S:=\{ ^N i_r(a,b):\hat \delta^{\tilde f}_r(a,b)\to ^N\mathbb F^{\tilde f}_r(a,b), a,b\in \mathbb  R \}$ 
 is called {\it collection of compatible splittings resp. compatible $N-$splittings} if $$\hat t_r(a,b)\cdot i_r(a,b)= i_r(a+2\pi, b+2\pi)\cdot \hat t_r(a,b)$$ or $$\hat t_r(a,b)\cdot \ ^Ni_r(a,b)= 
 ^N i_r(a+2\pi, b+2\pi)\cdot \hat t_r(a,b).$$
\end{definition}

Note that: 
\begin{enumerate}
\item The splitting $i_r(a,b)$ of $\pi_r(a,b) $ induce the splitting $^N i_r(a,b)$ of $^N\pi_r(a,b)$ by composition with the canonical projection $\mathbb F^{\tilde f}_r(a,b)\to ^N\mathbb F^{\tilde f}_r(a,b)= \mathbb F^{\tilde f}_r(a,b)/ TH_r(\tilde M).$ 
\item Collections of compatible splittings  and therefore of compatible $N-$splittings exist. It suffices to start with splittings for pairs $(a,b), 0\leq a < 2\pi$ and extend them for $a$ outside such interval by composing with the appropriate $\hat t^r$ and get compatible splittings and then derive from them compatible $N-$splittings.
\item 
The linear  maps $\pi_r(a,b) $ or $^N \pi_r(a,b)$ and the collection $\mathcal S$  of compatible splittings or compatible $N-$splittings, 
induce $$\pi_r\langle a,b\rangle : \mathbb F^{f}_r\langle a,b\rangle \to \hat \delta^{f}_r\langle a,b\rangle \ \rm {and} \  i_r\langle a,b\rangle:\hat \delta^{ f}_r\langle a,b\rangle \to \mathbb F^{ f}_r \langle a,b\rangle$$  or 
$$\pi_r\langle a,b\rangle: ^N\pi_r\langle a,b\rangle: ^N\mathbb F^{f}_r\langle a,b\rangle \to \hat \delta^{\tilde f}_r\langle a,b\rangle \ \rm {and}\  ^N i_r\langle a,b\rangle:\hat \delta^{f}_r\langle a,b\rangle \to ^N \mathbb F^{ f}_r \langle a,b\rangle.$$
\end{enumerate}
Item 3. requires some arguments.

To define $\pi_r\langle a, b\rangle$ and $^N \pi_r\langle a, b\rangle$ we notice first that the linear maps $\pi_r(a,b) $ and  $^N \pi_r(a,b)$  extend to $\sum_{i\in \mathbb Z} \mathbb F^{\tilde f}_r(a+2\pi i, b+2\pi i) \subseteq H_r(\tilde M)$ and $\sum_{i\in \mathbb Z} \ ^N \mathbb F^{\tilde f}_r(a+2\pi i, b+2\pi i)\subset ^N H_r(M;\xi)$ respectively. If this will be the case denote these extensions by $\overline \pi_r\langle a, b\rangle$ and $^N\overline \pi_r\langle a, b\rangle.$ 
To show this is the case it suffices to verify that if $x_i \in \mathbb F^f_r(a+2\pi i, b+ 2\pi i)$ and $\sum _{i\in \mathbb Z} x_i\in T(H_r(\tilde x)$ then $x_i\in {\mathbb F'}^f_r(B(a+2\pi i, b+ 2\pi i; 2\pi)) +T H_r(\tilde M);$ this guaranties that 
$\oplus_i \pi_r(a+2\pi i, b+2\pi i)(x_i)=0.$

Indeed since 
$$B(a+2\pi i, b+2\pi i; \infty) \cap (\bigcup _{j\in \mathbb Z, j\ne i, } B(a+2\pi j, b+2\pi j; \infty))\subseteq cB(a+2\pi i, b+2\pi i; 2\pi),$$ 
in view of Proposition \ref {P34} item  2.b, if 
$ x= \sum _{i\in \mathbb Z} x_i \in TH_r(\tilde M)$ with
 $x_i\in \mathbb F_r(a+2 \pi i, b+ 2\pi i)$ then  $$x_i = x- \sum _{i\ne j} x_j \in T H_r(\tilde X) +{\mathbb F'}^{\tilde f}_r(B(a+2\pi i, b +2\pi i; 2\pi)).$$ 
Define $\pi_r\langle a, b\rangle $ resp. $^N \pi_r\langle a, b\rangle $ to be the direct sum of $\overline \pi_r\langle a+ 2\pi i, b+2\pi i\rangle$ resp. of $^N\overline \pi_r\langle a+ 2\pi i, b+2\pi i\rangle$  over all $i\in \mathbb Z.$

Clearly the map $\pi_r\langle a,b\rangle$ is the factorisation of $\oplus_{k\in \mathbb \mathbb Z} \pi_r(a+2\pi k, b+2\pi k): \oplus_{k\in \mathbb \mathbb Z}\mathbb F^{\tilde f}_r(a+2\pi k, b+ 2\pi k)\to \hat\delta^f_r\langle a, b\rangle= \oplus \delta^{\tilde f}_r(a+2\pi k, b+2\pi k)$ by the projection $\pi : \oplus_{k\in \mathbb \mathbb Z}\mathbb F^{\tilde f}_r(a+2\pi k, b+ 2\pi k)\to 
\sum _{k\in \mathbb \mathbb Z}\mathbb F^{\tilde f}_r(a+2\pi k, b+ 2\pi k).$  A similar observation holds for $^N\pi_r\langle a,b\rangle.$

Define $i_r\langle a, b\rangle$  to be  the composition of $\oplus_{k\in \mathbb Z} i_r(a+2\pi k, b+2\pi k)$  with the projection $\pi : \oplus_{k\in \mathbb \mathbb Z}\mathbb F^{\tilde f}_r(a+2\pi k, b+ 2\pi k)\to 
\sum _{k\in \mathbb \mathbb Z}\mathbb F^{\tilde f}_r(a+2\pi k, b+ 2\pi k)$
and $^Ni_r\langle a, b\rangle$ to be the composition of  
$\oplus_{k\in \mathbb Z} \  ^Ni_r(a+2\pi k, b+2\pi k)$ with the projection 
$\pi : \oplus_{k\in \mathbb \mathbb Z}\  ^N\mathbb F^{\tilde f}_r(a+2\pi k, b+ 2\pi k)\to 
\sum _{k\in \mathbb \mathbb Z}\  ^N\mathbb F^{\tilde f}_r(a+2\pi k, b+ 2\pi k).$ 

Then a collection $\mathcal S$ of compatible splittings and implicitly of compatible $N-$splittings defines  the $\kappa-$linear map 
$$ \boxed{^\mathcal S I^N_r= \bigoplus_{\langle a,b\rangle \in \mathbb T} i_r\langle a,b\rangle: \bigoplus_{(a,b)} \hat\delta^{f}_r\langle a,b \rangle \to H^N_r(M; \xi_f)}\ ,$$  and for a small box $B$  the 
$\kappa-$linear maps 
  
$$\boxed{ ^\mathcal S I^N_r\langle B\rangle = \bigoplus_{\langle a,b\rangle \in \langle B \rangle} i^B_r\langle a,b): \bigoplus_{\langle a,b\rangle \in \langle B \rangle} \hat\delta^{\tilde f}_r\langle a, b \rangle \to ^N \mathbb F^{\tilde f}_r\langle B\rangle} \ .$$  

\begin{proposition}\label {P310}
Both $^\mathcal S I^N_r$ and $^\mathcal S I^N_r\langle B\rangle$ are $\kappa[t^{-1}, t]-$isomorphisms.
\end{proposition}

Indeed for a chosen collection $\mathcal S$ of compatible $N-$splittings consider   
the diagrams  

\begin{equation} \label {D4}
\xymatrix{ \bigoplus_{\langle\alpha,\beta\rangle} \hat \delta^{f}\langle \alpha, \beta \rangle \ar [rr]^{^\mathcal S I^N_r}  &&H^N_r(X;\xi)\\ 
\bigoplus _{\{\langle\alpha,\beta\rangle, \alpha\leq a, \beta\geq b\}}     \hat \delta^{ f} \langle \alpha, \beta \rangle \ar[u]^{\subseteq}\ar [rr]^{^\mathcal S I^N_r }  && ^N\mathbb  F^{ f}_r\langle a,b\rangle \ar[u]^{\subseteq}\\ 
\bigoplus _{\langle\alpha,\beta\rangle\in \langle cB\rangle} \hat \delta^{f}\langle \alpha, \beta\rangle \ar[u]^{\subseteq}\ar [rr]^{^\mathcal S I^N_r 
} && ^N {\mathbb F'}^{ f}_r \langle B\rangle \ar[u]^{\subseteq}} 
\end{equation}

and

\begin{equation}\label {D5}
\xymatrix{ 
\bigoplus _{\langle\alpha,\beta\rangle, \alpha\leq a, \beta\geq b}      \hat \delta^{ f} \langle \alpha, \beta\rangle \ar[d]^{\pi'_r}\ar [rr]^{^\mathcal S I^N_r }  && ^N\mathbb F^{ f}_r \langle a,b\rangle \ar[d]^{\pi''_r}\\
\bigoplus _{\langle\alpha,\beta\rangle\in \langle B\rangle} \hat \delta^{ f} \langle \alpha, \beta\rangle \ar [rr]^{^\mathcal S I^N_r\langle B\rangle} && {^N \mathbb F}^f_r\langle B\rangle
= \mathbb F^{ f}_r \langle B\rangle
.} 
\end{equation}
with $\pi'_r$ and $\pi''_r$ the obvious projections.

In view of Proposition \ref {P34} one has:

\hskip .1in $\hat \delta^{ f} \langle \alpha, \beta\rangle $
is a free $\kappa[t^{-1},t]-$module  with the multiplication by $t$ given by the isomorphism $\bigoplus_{k\in \mathbb Z} \hat t_r(\alpha+ 2\pi k, \beta +2\pi k),$ 

\hskip .1in the vector spaces involved in the above diagrams are all free $\kappa[t^{-1},t]-$modules  and  
in view of the commutativity of diagrams (\ref{E313}) and (\ref{D10}) all  arrows  in both diagrams are $\kappa[t^{-1},t]-$linear.  

\hskip .1in the horizontal arrows in the above diagrams are  isomorphisms, in particular so are $^\mathcal SI^N_r$'s and $^\mathcal SI^N_r\langle B\rangle .$ 
q.e.d

\begin{proposition} \label {P10} 

Both  $ ^N\mathbb F^{ f}_r\langle a,b\rangle $ and $^N {\mathbb F'}^{ f}_r \langle B\rangle $ are  split free submodules of $H^N_r(X;\xi),$ 
and ${\mathbb F}^{ f}_r \langle B\rangle $ is a quotient of split free submodules  hence also free. In particular $\hat \delta^f_r \langle a, b\rangle,$  which is canonically isomorphic to $\mathbb F^f_r \langle B(a,b;\epsilon)\rangle$ for $\epsilon <\epsilon (f),$ is a quotient of split free submodules
$\mathbb F_r^f \langle a, b\rangle/ 
{\mathbb F'}^f_r\langle B(a,b;\epsilon)\rangle.$  
\end{proposition}

{\it Definition of $\delta^f_r, \tilde \delta^f_r, \hat \delta^f_r, \hat {\hat  \delta}^f_r$  and $P^f_r(z)$}
\vskip .1in 
 In view of Proposition \ref {P1}, \ref{P2} , \ref{P37} and of Proposition \ref{P10} the assignments
 \begin{enumerate}
 \item $\mathbb T\ni \langle a, b\rangle \rightsquigarrow \delta^f_r \langle a, b\rangle,$  
 \item $\mathbb T\ni \langle a, b\rangle \rightsquigarrow \tilde \delta^f_r \langle a, b\rangle= (\mathbb F^f_r\langle a, b\rangle, {\mathbb F'}_r^f\langle B(a,b;\epsilon)\rangle),$  $\epsilon $ small.
 \item $\mathbb T\ni \langle a, b\rangle \rightsquigarrow \hat \delta^f_r \langle a, b\rangle,$
 \item $\mathbb T\ni \langle a, b\rangle \rightsquigarrow \hat {\hat \delta}^f_r \langle a, b\rangle,$ with $\hat {\hat \delta}^f_r \langle a, b\rangle$
 the von Neumann completion of $ {\hat \delta}^f_r \langle a, b\rangle$
\end{enumerate}
  defined by 
  \begin{enumerate}
 \item $ \delta^f_r \langle a, b\rangle:= \delta^{\tilde f}(a,b),$  
 \item $\tilde \delta^f_r \langle a, b\rangle:= (\mathbb F_r^f\langle a, b\rangle, 
{\mathbb F'}^f_r\langle B(a,b;\epsilon)\rangle,$ $\epsilon$ small enough,
 \item $\hat \delta^f_r \langle a, b\rangle = (\mathbb F_r^f\langle a, b\rangle / 
{\mathbb F'}^f_r\langle B(a,b;\epsilon)\rangle,$ $\epsilon$ small enough 
\end{enumerate}
 are configurations of points with multiplicity, of pairs of submodules of $H^N_r(X;\xi_f),$ 
 of free $\kappa[t^{-1},t]-$modules.

 We use the identification of $\mathbb T$ with $\mathbb C\setminus 0$ provided by the map $\langle a, b\rangle \to  z=e^{ia + (b-a)}$
 and if $z_1, z_2, \cdots z_k\in \mathbb C\setminus 0$ are the points in the support of $\delta^f_r,$  when regarded in $\mathbb C\setminus 0,$ define the polynomial 
 $$P_r^f(z):= \prod (z-z_i)^{ \delta^f_r(z_i)}.$$
  \vskip .1in  

When $\kappa=\mathbb C$ the von Neumann completion described in section \ref{S2} converts $\mathbb C[t^{-1},t]-$ into the von Neumann algebra $L^\infty(\mathbb S^1)$ and a $\mathbb C[t^{-1},t]-$valued inner product converts  $H^N_r(M;\xi)$ into a $L^\infty(\mathbb S^1)-$Hilbert module and ${\mathbb F'}^f_r \langle B\rangle,$  ${\mathbb F}^f_r\langle B\rangle$ and  $\hat \delta \langle a,b\rangle$ into Hilbert submodules.  The von Neumann completion  leads to the configuration  $\hat {\hat \delta}^f_r$ of mutually orthogonal $L^\infty(\mathbb S^1)-$Hilbert modules.  In case $X$ is an underlying space of a closed Riemannian  manifold or of a simplicial complex (hence a space equipped with a triangulation, the additional structure, the Riemannian metric or the triangulation provide such inner product.  
\vskip.1in

\section{ Proof of Theorems \ref{TT1}  and \ref{TT2}}

{\it Proof of Theorem \ref{TT1}.}

Item 1. follows from  Proposition \ref{P3} item 3. and the definitions of $ \delta^f_r, \hat \delta^f_r $ and $\hat {\hat \delta}^f_r .$
 
Item 2. follows from the fact that $\pi(r)\cdot \ ^\mathcal S I_r$ is an isomorphism, as established in Proposition \ref{P34} item 1. applied to $\tilde  f$ and from the definitions of $ \delta^f_r, \hat \delta^f_r $ and $\hat {\hat \delta}^f_r.$
The configuration $\hat\delta^f_r$ is derived from a configuration of pairs as described in section \ref{S2} with 
$\hat\delta^f_r \langle a,b\rangle= \mathbb F^f_r \langle a,b\rangle / {\mathbb F'}^f_r \langle B(a,b;\epsilon)\rangle$  for any $\epsilon <\epsilon (f).$

For Item 3.  one proceeds  as in the proof of Theorem 4.1 item 4. in \cite {B1} in case $X$ is a compact smooth manifold or a finite simplicial complex.  For example in case $X$ is a smooth manifold, possibly with boundary, 
any angle valued map is arbitrary closed to  a Morse angle valued map $f$ which takes different values on different critical points. Then the same remains true for $\tilde f:\tilde X\to \mathbb R,$ an infinite ciclic cover of this Morse map; this guarantees that for the sequence of  critical values $\cdots c_{i-1} < c_i  < c_{i+1} < \cdots ,$   the inclusion induced linear maps $H_\ast (\tilde X_{c_{i-1}}) \to H_\ast(\tilde X_{c_i})$ have co-kernel of dimension at most one.   As argued in the proof of Theorem 1 item 4. in \cite {B1}, this implies that $\delta^{\tilde f}_r$ and then $\delta^f_r$ takes only $0$ or $1$ as values.
 In the same way as in \cite{B1}, with the help of results on compact Hilbert cube manifolds e.g Theorem \ref{TT} below,  one derives Item 3. in Theorem \ref{TT1} in the generality stated.
\vskip .1in 

\noindent {\it Proof of Theorem \ref{TT2}.}

First observe that in view of  
Observation \ref{O} and Theorem \ref {TT} (stating results about Hilbert cube manifolds)  
it will suffice to prove Theorem \ref{TT2} for $X$ a finite simplicial complex.  

Indeed, if Theorem \ref {TT2} holds for any finite simplicial complex in view of Observation \ref {O} and Theorem \ref {TT} item 3. it holds for $K\times I^\infty$ hence by Theorem \ref {TT} item 2. for $X$ a compact Hilbert cube manifold and then again by Theorem \ref {TT} item 1. and Observation \ref {O}, holds for any compact ANR.  
\vskip .1in

For a continuous map $f:X\to \mathbb S^1$ and $K$ a compact space  denote by $f _K: X\times K\to \mathbb S^1$  the composition $f_K: = f\cdot \pi_X$ with $\pi_X: X\times K\to X$ the first factor projection.

\begin{obs} \label {O}
If $f:X\to \mathbb S^1$ is a continuous map with $X$ a compact ANR and $K$ is a contractible compact ANR then $\delta^f_r= \delta^{f_K}_r$ and $\hat \delta^f_r= \hat \delta^{f_K}_r.$ 
\end{obs} 
The statement follows in a straightforward manner from  the definitions of $\delta^f_r$ and $\hat\delta^f_r.$
\vskip .2in

Denote by $I^\infty$ the product of countable many copies of $I=[0,1]$ and write $I^{\infty}= I^k\times I^{\infty-k}.$ 

\begin{theorem} (Chapman, Edwards)  \cite{CH}  \label {TT}
\begin{enumerate}
\item
If $X$ is a compact ANR then $X\times I^\infty$ is a compact Hilbert cube manifold, i.e. locally homeomorphic to $I^\infty.$
\item Any  compact Hilbert cube manifold $M$  is homeomorphic to $K\times I^{\infty}$ for some $K$ finite simplicial complex.
\item If $K$ is a finite simplicial complex, $f:K\times I^{\infty} \to \mathbb S^1$ a continuous map and $\epsilon >0$ then there exists an $n$ and 
$g: K\times I^n\to \mathbb S^1$ a p.l-map such that $|| f- g_{I^{\infty-n}}|| <\epsilon$
\end{enumerate}
\end{theorem}
\vskip .1in

A proof of Items 1 and 2 can be found in  \cite{CH}. Item 3. is a rather straightforward consequence of the compacity of $K\times I^\infty$ and the approximation of continuous maps by p.l- maps  when the source is a finite simplicial complex (for more details  
cf. Proposition 6.5 in \cite {B}).

We proceed now with the verification of Theorem\ref{TT2} for $X$ a finite simplicial complex. 
 
In view of Observations 2.1, 2.2  and  2.3 the proof of item 1.  is the same as of Theorem 4.2 in \cite {B1} provided we replace
 $f :X\to \mathbb R,$  by   $\tilde f: \tilde X\to R$ a lift of $f:X\to \mathbb S^1$ representing $\xi.$  The basic ingredient,
Proposition 3.16, (based on Lemma 3.17 and Lemma 3.18)  in \cite{B1} holds for $h:Y\to \mathbb R,$ $Y$ a locally compact ANR and $h$ a proper map,  instead of  $f :X\to \mathbb R,$ continuous with $X$ compact. 

For the reader's convenience  we restate this proposition in the way it will to be used but  beforehand we introduce the  notation $$D(a,b; \epsilon):= B(a+\epsilon, b+\epsilon; 2\epsilon)= (a-\epsilon, a+\epsilon]\times ((b-\epsilon, b+\epsilon].$$  
\begin{proposition}\label {P316}

Let $f:X\to \mathbb S^1$ be a  tame map and $\epsilon <\epsilon(f)/3.$ For any  map $g:X\to \mathbb S^1$  which satisfies  $|| f- g ||_\infty <\epsilon$ and $a,b$ critical values  of a lift $\tilde f:\tilde X \to \mathbb R$ of $f$ one has:

\begin{equation}\label{E3}
 \quad   \sum_{x\in D(a,b;2\epsilon)} \delta^{\tilde g}_r(x)=  \delta ^{\tilde f}_r(a,b), 
\end{equation}

\begin{equation} \label{E40}
 \quad \quad  \supp \  \delta^{\tilde g}_r \subset \bigcup  _{(a,b)\in \supp\ \delta^{\tilde f}_r} D(a,b;2\epsilon) 
\end{equation}
when $\tilde g:\tilde X\to \mathbb R$ is any  lift of $g.$ 

Moreover if $\kappa= \mathbb C$ and  $H_r(\tilde X)$ is equipped with a Hermitian scalar product the above statement can be strengthen to 
\begin{equation}\label{E5}
 x\in D(a,b;2\epsilon) \Rightarrow \hat \delta^{\tilde g}_r(x)\subseteq {\hat \delta} ^{\tilde f}_r(a,b), 
 \  \bigoplus _{x\in D(a,b;2\epsilon)} \hat \delta^{\tilde g}_r(x)= \hat \delta^{\tilde f}_r(a,b).
 \end{equation}
with  $\hat {\delta}^{\tilde f}_r(x)\perp \hat {\delta}^{\tilde f}_r(y).$
Here $\hat {\delta}^{\tilde f }_r(x)$ or $\hat {\delta}^{\tilde g }_r(x)$ denotes the orthogonal complement of ${\mathbb F'}^{\tilde f }_r (D(a,b;\epsilon))$ in ${\mathbb F}^{\tilde f}_r(a,b), $ or of ${\mathbb F'}^{\tilde g }_r (D(a,b;\epsilon))$ in ${\mathbb F}^{\tilde g}_r(a,b),$\
for $\epsilon$ small enough. 
\end{proposition}

The steps in the proof of Item 1. of Theorem \ref{TT2} are similar to the steps described in subsection 4.2 in \cite{B1}. We  summarize them below.  
 
\begin{enumerate}
\item  For a pair $(X,\xi),$ $X$ compact ANR, 
let $C_\xi(X, \mathbb S^1)$  denote the set of maps in the homotopy class defined by $\xi$ equipped with the compact open topology.
Note, in view of Observation \ref{O21}, that: 
\begin{enumerate}
\item the compact open topology  is  induced from the complete metric   
$D(f,g)$ and $D(f,g)= D(\tilde f, \tilde g)$  for appropriate liftings. 
 \item for $f,g\in  C_\xi (X,\mathbb S^1)$ 
 and any  sequence $0=t_0 <t_1 \cdots t_{N-1} <t_N=1,$ by Observation \ref {O21} item 3.,
the canonical homotopy $f_t$  from $f$ to  $g,$  i.e. $f_0= f, f_1= g,$  satisfies 
\begin{equation}\label {E0}
D(f,g)\geq  \sum _{0\leq i <N} D(f_{t_{i+1}}, f_{t_i}).
\end{equation}  
\end{enumerate}
\item For  $X$ is a simplicial complex let  $\mathcal U\subset  C_\xi (X,\mathbb S^1)$ be the subset  of p.l. maps. One can verify that:
\begin{enumerate}
\item $\mathcal U$ is a dense subset in $C_\xi (X,\mathbb S^1),$
\item  if $f,g\in \mathcal U$  with $D(f,g)<\pi$ 
then, for the canonical homotopy $f_t$ each $f_t \in \mathcal U,$ hence  $ \epsilon(f_t) >0.$ Then for any $t\in[0,1]$ there exists $\delta (t)>0$ s.t.
$t',t''\in (t-\delta(t), t+\delta(t))$ implies 
$D(f_{t'}, f_{t''})<\epsilon (f_t)/3.$
\end{enumerate}
Both statements (a) and (b) are argued as  in \cite{B1}.
\item Consider the space of configurations $Conf_{b_r}(\mathbb T),$ $b_r= \beta^N_r(X;\xi)$ viewed as $S^{b_r}(\mathbb T),$ the $b_r$ fold symmetric product of $\mathbb T$ 
equipped with the induced metric, $\underline D$,  
which is  complete.
Since any map in $\mathcal U$ is tame,  in view 
 Proposition (3.16) in \cite{B1},  $f, g \in \mathcal U$ with $D(f,g) <\epsilon(f)/3$ imply
\begin{equation}\label {EE14}
\underline D(\delta^f_r, \delta^g_r) \leq 2 D(f,g).
\end{equation} 
\end{enumerate}
This suffices to conclude the continuity of the assignment $f\rightsquigarrow \delta^f_r.$
\vskip .1in

To finalize the proof of Theorem (\ref{TT2}) item 1. we check 
first  (step 1.) that the inequality (\ref{EE14}) extends to all $f,g \in \mathcal U,$   
second (step 2.) that the inequality (\ref{EE14}) extends to all $f,g \in C_\xi (X,\mathbb S^1)$  for $X$ a finite simplicial complex, 
third (step 3.) that the inequality (\ref{EE14}) extends to all $f,g \in C_\xi (X,\mathbb S^1)$  for $X$ an arbitrary compact ANR. 

\begin {itemize}
\item {\it Step 1.:}

Start with $f,g\in \mathcal U$  and consider the canonical homotopy  $\tilde f_t= t \tilde f +(1-t)\tilde g,$ $t\in [0,1]$ between two lifts  lifts $\tilde f, \tilde g $ of $f$ and $g$ which satisfy $D(f,g)= D(\tilde f, \tilde g).$   Note that each $\tilde f_t$ satisfies $\tilde f_t(\mu(n,x))= \tilde f_t(x) + 2\pi n,$  hence is a lift of a p.l. map $f_t.$ 

Choose a sequence $0<t_1< t_3 <t_5,\cdots t_{2N-1} <1$  
such that 
for $i=1,\cdots, (2N-1)$ the intervals $(t_{2i-1}-\delta(t_{2i-1}), t_{2i-1}+\delta(t_{2i-1})),$ with $\delta(t_i)$ as in 2.(b), cover $[0,1]$ and 
$(t_{2i-1},t_{2i-1}+\delta(t_{2i-1}))\cap  (t_{2i+1}-\delta(t_{2i+1}), t_{2i+1})\ne \emptyset.$    This is possible in view of the compacity  of $[0,1].$ 

Take $t_0=0, t_{2N}=1$ and $t_{2i}\in (t_{2i-1},t_{2i-1}+\delta(t_{2i-1}))\cap  (t_{2i+1}-\delta(t_{2i+1}), t_{2i+1}).$ 
To simplify the notation abbreviate  $ f_{t_i}$ to $f_i.$
In view of  2.  and  3. above (inequality (\ref{EE14}))  one has: 

$|t_{2i-1}- t_{2i}|<\delta(t_{2i-1})$ implies  $\underline D(\delta^{f_{2i-1}},\delta^{f_{2i}})< 2D(f_{2i-1}, f_{2i})$ and 

$|t_{2i}- t_{2i+1}|<\delta(t_{2i+1})$ implies  $\underline D(\delta^{f_{2i}},\delta^{f_{2i+1}} )< 2D(f_{2i}, f_{2i+1})$
\newline Then we have $$\underline D(\delta^f, \delta^g) \leq \sum_{0\leq i<2N-1} \underline D(\delta^{f_{i}}, \delta^{f_{i+1}})  \leq 2\sum_{0\leq i<2N-1}  D(f_i, f_{i+1})\leq 2 D(f,g)$$ in view of Observation \ref{O21} item 3. 

\item  {\it Step 2.:} Suppose $X$ is a simplicial complex. In view of the density of $\mathcal U$  and of the completeness of the metrics on $C_\xi(X;\mathbb S^1)$ and $Conf_{b_r}(\mathbb T),$ the inequality (\ref {EE14}) extends to the entire $C_\xi(X;\mathbb S^1).$ 
Indeed the assignment $\mathcal U\ni f\rightsquigarrow \delta^f_r\in C_{b_r}(\mathbb R^2)$ preserve the Cauchy sequences.

\item {\it Step 3.:} We verify the inequality (\ref {EE14}) for $X= K\times I^\infty,$ $K$ simplicial complex and $I^\infty$ the Hilbert cube. One proceed exactly as in \cite{B1}.
Since by Theorem \ref{TT2} item 2 any compact Hilbert cube manifold is homeomorphic to $K\times I^\infty$ for some finite simplicial complex $K$, the inequality (\ref {EE14})
 continues to hold. 
 Since for $X$ a compact ANR 
   
(i)  $X\times I^\infty$ is a Hilbert cube manifold by Theorem \ref{TT2} item 1, 

(ii)  $I: C(X;\mathbb R)\to C(X\times I^\infty;\mathbb R)$ defined by $I(f)= \overline f_{I^\infty}$ is an isometric embedding and 

(iii)  $\delta^f= \delta^{\overline f_{I^\infty}},$  

the inequality (\ref{EE14}) holds for 
 $X$  a compact ANR.
\end{itemize}
\vskip .1in 

To check item 2. in Theorem \ref{TT2} we begin with a few observations.
If $\kappa= \mathbb C,$ a Riemannian metric on a closed smooth manifold $ M^n=X,$ or a triangulation of a compact ANR  $X,$   
provides a Hermitian scalar product on $H_r(\tilde X)$ 
invariant to the action of the group of deck transformations  of the covering $\tilde X\to X.$ Ultimately this provides a $\mathbb C[t^{-1},t]-$compatible Hermitian inner product on $H^N_r(X;\xi)$ and then a collection of compatible $N-$splittings, $^N i_r(a,b)$ $(a,b)\in \mathbb R^2,$ and then the collection of compatible $N-$splittings, $^Ni_r\langle a, b\rangle,$ $\langle a, b\rangle\in \mathbb T,$  for both $f$ and $g.$  The images of these splittings are the free 
submodules ${\hat \delta}^f$ and ${\hat \delta}^g.$ 
In view of Proposition \ref{P316} for a given $f,$ $(a,b)\in Cr(f)\times Cr(f),$ $\epsilon <\epsilon (f)$  and any $g$ with $||g- f||_{\infty}< \epsilon/3$  the following two subspaces of spaces of $H^N_r(X;\xi),$
$$\sum _{(a',b')\in D(a,b;\epsilon)\cap \supp \delta^{\tilde g}} \ ^N i_r\langle a', b'\rangle (\hat \delta^g\langle a', b'\rangle)\ \rm {and} \  \  ^N i_r\langle a, b\rangle (\hat \delta^f\langle a, b\rangle),$$  are equal. 

The $\mathbb C[t^{-1},t]-$compatibility permits to pass to von Neumann completions and derive the collection of Hilbert submodules $\hat{\hat \delta}^f$ and $\hat{\hat \delta}^g$ 
which under the above hypotheses satisfy $$\sum _{(a',b')\in D(a,b;\epsilon)\cap \supp \delta^{\tilde g}} 
\hat{\hat \delta}^g\langle a', b'\rangle)=
\hat{\hat \delta}^f\langle a, b\rangle.$$ 
This implies the continuity of the assignment $f\rightsquigarrow \hat{\hat \delta}^f_r$ when the space of configurations is equipped with the fine topology and then with the natural topology, hence Item 2. of Theorem \ref{TT2}.

\section{ Proof of  Theorem \ref {TT3}}\label {S5}

  We prove Theorem \ref{TT3}  for weakly tame maps $f:M\to \mathbb S^1,$ $M$ closed topological manifold of dimension $n,$ whose set of non topological regular values is finite.
 If the set of such maps is dense in the set of all maps equipped with compact open topology then, in view of Theorem \ref{TT2}, the result holds for all continuous maps. One expects  this be always the case. When the manifold is homeomorphic to a finite simplicial complex this is indeed the case  since  a p.l map is weakly tame and has finitely many critical values  and the the set of p.l. maps is dense in the set of  all continuous maps with compact open topology. For the manifolds which have no triangulation a possible argument  for such density is considerably longer and will not be provided in this paper. We were unable to locate a reference in the literature.  
 For $f$ a weakly tame map it will suffices to consider only regular values $a,b$. This because for arbitrary pairs $c,c'$ one can find  $\epsilon  >0$ small enough so that for $a'= c-\epsilon, a= c+\epsilon, b= c'-\epsilon, b'=c+\epsilon, B=(a',a]\times [b, b')$ in view of Proposition \ref{P3} one has 
 \begin{equation*}
 \begin{aligned}
 \hat \delta_r^{\tilde f}(c,c')=  \hat \delta_r^{\tilde f}(a,b) = \mathbb F^{\tilde f}(B)
 \hat \delta_r^{\tilde f}(c',c)=  \hat \delta_r^{\tilde f}(b',a')= \mathbb F^{\tilde f}(B').
 \end{aligned}
 \end{equation*}   
  
The proof  of Theorem \ref{TT3} requires some additional notations and considerations.

\vskip .1in
{\it Some additional notation and definitions}
 
 Recall that a topologically regular value is a value $s$ which has a neighborhood $U$ s.t. $f: f^{-1}(U)\to U$ is a topological bundle. If so any lift (infinite cyclic cover) $\tilde f: \tilde M\to \mathbb R$ of $f$ has the set of critical values discrete and 
$2\pi-$periodic.

We use the notations: 
\begin{enumerate}
\item 
$\tilde M_a:= \tilde f^{-1}((-\infty, a]),$  \quad  $\tilde M^a:= \tilde f^{-1}([a, \infty)), a\in \mathbb R,$
\item 
$\mathbb I_a(r) :=\img (H_r(\tilde M_a) \to H_r(\tilde M)),$ \quad $\mathbb I^a(r) :=\img (H_r(\tilde M^a) \to H_r(\tilde M)),$ 
\item  $\mathbb F^{\tilde f}_r(a,b)= \mathbb I^{\tilde f}_a\cap \mathbb I^b_{\tilde f}$, \quad $i_r(a,b):  \mathbb F^f_r(a,b) \subset  H_r(\tilde M)$ the  inclusion.

\noindent In addition consider :
\item 
$\mathbb G^{\tilde f}_r(a,b):= H_r(\tilde M)/ (\mathbb I^{\tilde f}_a + \mathbb I^b_{\tilde f}),$ \quad 
$p_r(a,b): H_r(\tilde M)\to \mathbb G^{\tilde f}_r(a,b)$ the  canonical projection.
\item For a box $B= (a',a]\times [b,b')$ denote by:
 \begin{enumerate}
\item $\mathbb F^{\tilde f}_r(B):= \coker (\mathbb F^{\tilde f}_r(a',b)\oplus \mathbb F^{\tilde f}_r(a,b')\to \mathbb F^{\tilde f}_r(a,b))$ and $\pi^B_r: \mathbb F^{\tilde f}_r(a,b)\twoheadrightarrow \mathbb F^f_r(B)$ the canonical surjection,   
\item 
$\mathbb G^{\tilde f}_r(B):= \ker (\mathbb G^{\tilde f}_r(a',b')\to \mathbb G^{\tilde f}_r(a',b)\times_{\mathbb G^{\tilde f}_r(a,b)} \mathbb G^{\tilde f}_r(a,b'),$  and  $u^B_r: \mathbb G^{\tilde f}_r(B)
\rightarrowtail \mathbb G^{\tilde f}_r (a',b')$ the canonical inclusion.   
\item Since $\mathbb F^{\tilde f}_r(B)$ identifies canonically to  $\mathbb I_a(r)\cap \mathbb I^b(r) / (\mathbb I_{a'}(r)\cap \mathbb I^b(r) +\mathbb I_a(r)\cap \mathbb I^{b'}(r))$  and $\mathbb G^{\tilde f} (B) $ identifies canonically to $(\mathbb I_{a'}(r)+ \mathbb I^b(r)) \cap (\mathbb I_{a}(r)+ \mathbb I^{b'}(r))/ (\mathbb I_{a'}(r)+ \mathbb I^{b'}(r))$  then the inclusion 
\newline $\mathbb I_a(r)\cap \mathbb I^b(r)   \subseteq (\mathbb I_{a'}(r)+ \mathbb I^b(r)) \cap (\mathbb I_{a}(r)+ \mathbb I^{b'}(r))$
induces the  linear map 
\newline $\theta_r(B): \mathbb F^{\tilde f}_r(B)\to \mathbb G^{\tilde f}_{r}(B)$  which is an isomorphism. 
\end{enumerate} For a verifications one can consult \cite {B1} Proposition 4.7.      
\end{enumerate}

If $a$ is a topologically regular value  then $\tilde M_a$ and $\tilde M^a$ are manifolds with compact boundary $f^{-1}(a)$ and  denote by

\begin{equation}\label{E23}
\begin{aligned}
H^{BM}_r(\tilde M)=&\varprojlim_{\tiny 0< l, t\to \infty} H_r(\tilde M, \tilde M_{-l}\sqcup \tilde M^t),\\
H^{BM}_r(\tilde M_a)=&\varprojlim_{\tiny 0< l \to \infty} H_r(\tilde M_a, \tilde M_{a-l}),\\
H^{BM}_r(\tilde M^a)=&\varprojlim_{0< l\to \infty} H_r(\tilde M, \tilde M^{a+l}),\\
H^{BM}_r(\tilde M, \tilde M_a)=&\varprojlim_{0< l\to \infty} H_r(\tilde M, \tilde M_a\sqcup \tilde M^{a+l}),\\
H^{BM}_r(\tilde M, \tilde M^a)=&\varprojlim_{0< l\to \infty} H_r(\tilde M, \tilde M^a\sqcup \tilde M_{a-l}).
\end{aligned}
\end{equation}

The reader will recognize on  left side of the equalities (\ref{E23}) the notation for the Borel-Moore homology vector spaces with coefficients in $\kappa,$ the right homology to extend the Poincar\'e Duality from compact manifolds  to arbitrary finite dimensional manifolds. 
\vskip .1in
{\it Poincar\'e duality diagrams for $\tilde M$}

One has the following commutative diagrams whose vertical arrows are isomorphisms, diagrams referred below as the Poincar\'e Duality diagrams for non closed  manifolds.
\vskip .2in 
\begin{equation}\label{PD1}
\xymatrix{&H^{BM}_r( \tilde M_a)\ar[d]^{PD^1_a}\ar[rr]^{i_a(r)}  \quad &&H^{BM}_r( \tilde M)\ar[d]^{PD}\ar[rr]^{j_a(r)} &&H^{BM}_r( \tilde M, \tilde M_a)\ar[d]^{PD^2_a}&\\
&H^{n-r}( \tilde M, \tilde M^a)\ar[d]^{=} \ar[rr]^{s^a(n-r)} \quad&&H^{n-r}( \tilde M)\ar[d]^{=} \ar[rr]^{r^a(n-r))} &&H^{n-r}( \tilde M^a)\ar[d]^{=}&\\
&(H_{n-r}( \tilde M, \tilde M^a))^\ast \ar[rr]^{(j^a(n-r))^\ast} \quad&&H_{n-r}( \tilde M)^\ast \ar[rr]^{(i^a(n-r))^\ast} &&H_{n-r}( \tilde M^a)^\ast&}
\end{equation}
\vskip .2in
\begin{equation}\label {PD1'}
\xymatrix
{&H^{BM}_r( \tilde M^b)\ar[d]^{PD_1^b} \ar[rr]^{i^b(r)} \quad &&H^{BM}_r( \tilde M)\ar[d]^{PD}\ar[rr]^{j^{b}(r)} &&H^{BM}_r( \tilde M, \tilde M^b)\ar[d]^{PD_2^b}&\\
&H^{n-r}( \tilde M, \tilde M_b)\ar[d]^{=} \ar[rr]^{s_b(n-r)}  \quad &&H^{n-r}( \tilde M)) \ar[d] ^{=}\ar[rr]^{r_b(n-r)} &&H^{n-r}( \tilde M_b)\ar[d]^{=}\\
&(H_{n-r}( \tilde M, \tilde M_b))^\ast \ar[rr]^{(j_b(n-r))^\ast} \quad &&(H_{n-r}( \tilde M))^\ast \ \ar[rr]^{(i_b(n-r))^\ast} &&(H_{n-r}( \tilde M_b))^\ast &.}
\end{equation}

One can derive these diagrams from the Poincar\'e duality for compact bordisms   
$(\tilde f^{-1}[a,b], \tilde f^{-1}(a),f^{-1}(b)))$ when $a$ and $b$ are topologically regular values, by passing to the limit  $a\to -\infty$ or $b\to \infty$ with no knowledge about Borel--Moore homology.

The Poincar\'e Duality isomorphism, $$\xymatrix {PD^{BM}_r : H^{BM}_r( \tilde M)\ar[r]^-{PD_r}& H^{n-r}( \tilde M)\ar[r]^{=}&(H_{n-r}( \tilde M))^\ast} $$ we consider below is the composition of the vertical arrows in the middle of diagram (\ref{PD1}) or (\ref{PD1'}).

Note that all three terms of this sequence are $\kappa[t^{-1},t]-$modules and the two arrows are $\kappa[t^{-1},t]-$linear with the multiplication by $t$ given the the linear isomorphism induced by the deck transformation $\tau_r.$

If one uses $H^{BM}_r(\cdots)$ instead of $H_r(\cdots)$  one can also consider  $ ^{BM}\mathbb F^{\tilde f}_r(a,b),$   $ ^{BM}\mathbb F^{\tilde f}_r(B)$ and $ ^{BM}\hat \delta^{\tilde f}_r(a,b)$ instead of $ \mathbb F^{\tilde f}_r(a,b),$ $\mathbb F^{\tilde f}_r(B)$ and $\hat \delta^{\tilde f}_r(a,b).$ Proposition  \ref{P52} item 3. below will show that $^{BM}\mathbb F^{\tilde f}_r(B)$ and $^{BM}\hat \delta^{\tilde f}_r(a,b)$ are  canonically isomorphic to $\mathbb F^{\tilde f}_r(B)$ and $\hat \delta^{\tilde f}_r(a,b).$
\vskip .1in
{\it Intermediate results}

With the definitions already given one has the following proposition. 

\begin{proposition}\ \label {P51}
\begin{enumerate}
\item  For any $a,b$ regular values of $\tilde f$ the Poincar\'e Duality isomorphism restricts to an isomorphism 
$$PD^{BM}_r (a,b) : ^{BM}\mathbb F^{\tilde f}_r(a,b)\to (\mathbb G^{\tilde f}_{n-r}(b,a))^\ast.$$ 
\item For any box $B= (a',a]\times [b,b')$ and $B'= (b,b']\times [a',a))$  with all $a, a', b, b'$ regular values $PD^{BM}_r$  induces the isomorphisms $PD^{BM}_r(a,b),$ $PD^{BM}_r (B),$ 
making the diagram below commutative.
\end{enumerate} 
\end{proposition}

\begin{equation} \label {E13}
\xymatrix {^{BM}\mathbb F^{\tilde f}_r(a',b')\ar[d]\ar[rr]^{PD^{BM}_r(a',b')}&&(\mathbb G^{\tilde f}_{n-r}(b',a'))^\ast\ar[d]&&\\
^{BM}\mathbb F^{\tilde f}_r(a,b)\ar[dd]\ar[dr]\ar[rr]^{PD^{BM}_r}&&(\mathbb G^{\tilde f}_{n-r}(b,a))^\ast\ar[dr]^{u^\ast _{n-r}}\ar[dd]^<<<{p_r^\ast}&&\\
&^{BM}\mathbb F^{\tilde f}_r(B)\ar[rr]^{PD^{BM}_r(B)}&&(\mathbb G^{\tilde f}_{n-r}(B'))^\ast\ar[r]^{\theta_{n-r}^\ast}&(\mathbb F^{\tilde f}_{n-r}(B')^\ast\\
H^{BM}_r(\tilde M)\ar[rr]^{PD^{BM}_r}&&(H_{n-r}(\tilde M))^\ast &&}
\end{equation}

\begin{proof}

Item 1.:  In view of diagrams (\ref {PD1}) and (\ref{PD1'}) one has $img \ i_a(r) \cap img \ i^b(r)=\ker j_a(r) \cap \ker j^b(r)\simeq 
\ker (i^a(n-r))^\ast \cap \ker (i_b(n-r))^\ast\simeq (\coker (i_b(n-r) \oplus i^a(n-r))^\ast =( \mathbb G^{\tilde f}_{n-r}(b,a))^\ast.$ The first  equality holds  by exactness of the first rows in the diagrams, the second by the equality of the top and bottom  right horizontal arrows and the third  by linear algebra duality and the fourth by the definition of $\mathbb G_{n-r}.$  

Item 2.: Consider the box $B= (a',a]\times [b,b')$ and denote by $B'$ the box $B'= (b,b']\times [a',a).$
Note that the image of the diagram $$
^{BM}\mathcal F(B):= \begin{cases}\xymatrix{ ^{BM}\mathbb F^{\tilde f}_r (a',b')\ar[r]\ar[d]&^{BM}\mathbb F^{\tilde f}_r (a,b')\ar[d]\\
^{BM}\mathbb F^{\tilde f}_r (a',b)\ar[r]&^{BM}\mathbb F^{\tilde f}_r (a,b)}\end{cases}$$ 

by $^{BM}PD_r $ is the diagram  

$$
\mathcal G(B')^\ast:=
 \begin{cases}\xymatrix{ (\mathbb G^{\tilde f}_{n-r} (b',a'))^\ast \ar[r]\ar[d]&(\mathbb G^{\tilde f}_{n-r} (b',a))^\ast\ar[d]\\
(\mathbb G^{\tilde f}_r (b,a'))^\ast \ar[r]&(\mathbb G^{\tilde f}_{n-r} (b,a))^\ast}\end{cases}$$

which is the dual of the diagram 
$$\mathcal G(B'):= \begin{cases}\xymatrix{ \mathbb G^{\tilde f}_{n-r} (b,a) \ar[r]\ar[d]&(\mathbb G^{\tilde f}_{n-r} (b',a)\ar[d]\\
\mathbb G^{\tilde f}_r (b,a') \ar[r]&(\mathbb G^{\tilde f}_{n-r} (b',a')}\end{cases}$$

Therefore $^{BM}PD_r$ induces an isomorphism from $^{BM}\mathbb F^{\tilde f}_r(B)= \coker ^{BM}\mathcal F(B)$ to $(\ker (\mathcal G(B'))^\ast= (\mathbb G^{\tilde f}_{n-r}(B'))^\ast.$

\end{proof}
\vskip .1in 
From diagram (\ref {E13}) one derives 

\begin{equation} \label {D14}
\xymatrix {^{BM}\mathbb F^{\tilde f}_r(a,b)\ar[dd]\ar[dr]\ar[rr]^{PD^{BM}_r(a,b)}&&(\mathbb G^{\tilde f}_{n-r}(b,a))^\ast\ar[dr]^{u^\ast _{n-r}}\ar[dd]^<<<<{p_r^\ast(a,b)}&&\\
&^{BM}\hat\delta^{\tilde f}_r(a,b)\ar[rr]&&(\hat\delta^{\tilde f}_{n-r}(b,a))^\ast\\
H^{BM}_r(\tilde M)\ar[rr]^{PD^{BM}_r}&&(H_{n-r}(\tilde M))^\ast &&}
\end{equation}

with the horizontal arrows isomorphisms, the vertical arrows injective and the oblique arrows surjective.

Indeed for  $B= (a-\epsilon,a+\epsilon]\times [b-\epsilon, b+\epsilon)$ in which case $B'= (b-\epsilon, b+\epsilon]\times [a-\epsilon, a+\epsilon)$
and $\epsilon$  small enough to have (in view of Proposition 3.4)  $\hat \delta^{\tilde f}_r(a,b)= \mathbb F^{\tilde f}_r(B)$ and $\hat \delta^{\tilde f}_{n-r}(b,a)= \mathbb F^{\tilde f}_r(B')$ 
 the diagram (\ref{E13}) gives rise to  the diagram (\ref{D14}).

\vskip .2in 

The key observation for   finalizing  items 1. and 2. is the following proposition.

\begin{proposition} \label {P52}\ 
The $\kappa-$ linear maps \ $\mathbb F^{\tilde f}_r (a,b)\to ^{BM} \mathbb F^{\tilde f}_r (a,b) :$ 
\begin{enumerate}
\item
are compatible with  the deck transformations,
\item
are surjective, 
\item
have the kernel $C_r(M)$ independent on $(a,b),$ equal to the kernel of the $\kappa[t^{-1},t]-$linear map $H_r(\tilde M)\to H_r^{BM}(\tilde M)$ and equal to $T(H_r(\tilde M)$. 

\end{enumerate}
\end{proposition}
\begin{proof}
One shows first that one has  a natural short exact sequence $$0\to C_r(M) \to \mathbb F^{\tilde f}_r(a,b) \to ^{BM}\mathbb F^{\tilde f}_r(a,b)\to 0$$   which is compatible with the action provided by the deck transformations and leaving $C_r(M)$ fixed, and second that  
$C_r(M)$ is exactly the $\kappa[t^{-1}, t]-$torsion of the $H_r(\tilde M).$  Precisely one show that one  has the following commutative diagram with exact sequences as rows and $C_r(M)= \mathbb I^{\tilde f}_{-\infty}(r) + \mathbb I^\infty_{\tilde f}(r).$ 

\begin{equation}\label {E14}
\xymatrix {0\ar[r]& C_r(M) \ar[r]\ar[d]^{t_r= id} &\mathbb F^{\tilde f}_r(a,b)\ar[r]\ar[d]^{t_r} & ^{BM}\mathbb F^{\tilde f}_r(a,b)\ar[r]\ar[d]^{t^{BM}_r}& 0\\
0\ar[r]& C_r(M) \ar[r] &\mathbb F^{\tilde f}_r(a+2\pi ,b+2\pi)\ar[r] & ^{BM}\mathbb F^{\tilde f}_r(a+2\pi,b+2\pi)\ar[r]& 0.}
\end{equation}

The proof uses diagram (\ref {E00}) below where $-l < a' <a$ and $b<b' <t.$  In this diagram the vertical columns  are exact sequences
By passing to limits when $l, t \to \infty,$ diagram (\ref{E00}) induces diagram (\ref{E01}) which  provides the relation between $\mathbb F_r (a,b), \mathbb F_r(a',b'), H_r(\tilde M)$ and their Borel-Moore versions. 
\vskip .2in
 
\begin{equation}\label {E00}
\xymatrix{
H_{r-1}(\tilde M_{-l})\ar[r]^{=} &H_{r-1}(\tilde M_{-l})\ar[r]^{\hat i_{-l}(r-1)} &H_{r-1}(\tilde M_{-l}\sqcup \tilde M^t)&H_{r-1}(\tilde M^t)\ar[l]_{\hat i^t(r-1)}&H_{r-1}(\tilde M^t)\ar[l]_{=}\\
H_r(\tilde M_{a'},\tilde M_{-l})\ar[r]\ar[u]  &H_r(\tilde M_a, \tilde M_{-l})\ar[r]\ar[u] &H_r(\tilde M, \tilde M_{-l}\sqcup \tilde M^t)\ar[u]  &H_r(\tilde M^b, \tilde M^t)\ar[l]\ar[u] &H_r(\tilde M^{b'}, \tilde M^t)\ar[l]\ar[u]       \\
H_r(\tilde M_{a'})\ar[r]\ar[u] \ar@/^1pc/[rr]^<<<{i_{a'}(r)}&H_r(\tilde M_a)\ar[r]_{i_a(r)}\ar[u]  &H_r(\tilde M)\ar[u] &H_r(\tilde M^b)\ar[l]^{i^b(r)}\ar[u]&H_r(\tilde M^{b'})\ar[l]\ar@/_1pc/[ll]_<<<{i^{b'}(r)} \ar[u]\\                                                                              
H_r(\tilde M_{-l})\ar[r]^{=}\ar[u] &H_r(\tilde M_{-l})\ar[r]^{\hat i_{-l}(r)}\ar[u] &H_r(\tilde M_{-l}\sqcup \tilde M^t)\ar[u]&H_r(\tilde M^t)\ar[l]_{\hat i^t(r)}\ar[u]&H_r(\tilde M^t)\ar[l]_{=}\ar[u]}
\end{equation}.

\begin{equation} \label {E01}
\xymatrix 
{H^{BM}( \tilde M_{a'})\ar[r]\ar@/^1pc/[rr]^{i_{a'}^{BM}(r)}& H^{BM}_r( \tilde M_a)\ar[r]_{i^{BM}_a(r)} &H_r^{BM}(\tilde M) & H^{BM}_r(\tilde M^b)\ar[l]_{(i^{BM})^b(r)} & H^{BM}(\tilde M^{b'})\ar[l]\ar@/^1pc/[ll]^{(i^{BM})^{b'}(r)}
\\
H( \tilde M_{a'})\ar[r]\ar[u]& H_r( \tilde M_a)\ar[u]\ar[r]^{i_a(r)}&H_r(\tilde M)\ar[u]  & H_r(\tilde M^b)\ar[u] \ar[l]_{i^b(r)} & H_r(\tilde M^{b'})\ar[l]\ar[u]}
\end{equation}
Diagram (\ref{E01}) leads to  the linear map $\mathbb F^{\tilde f}_r (a,b)\to ^{BM} \mathbb F^{\tilde f}_r (a,b)$  and establishes the compatibility with the deck transformations, hence Item 1.  of Proposition \ref{P52}.

\vskip .1in

Since $^{BM}\mathbb F^{\tilde f}_r (a,b) =\img ((i^{BM}_a(r))  \cap \img ((i^{BM})^b(r))$ 
and \  $img(\hat i_{-l}(r)\cap img(\hat i^{t}(r))=0$ for any $r, l, t,$  a careful analysis of the projective limit  
and  of the diagram (\ref{E00}) 
implies that
$$\mathbb F^{\tilde f}_r(a,b)\to ^{BM}\mathbb F_r^{\tilde f}(a,b)$$ 
is surjective, (hence Item 2. holds), with kernel  isomorphic to  
$$\varprojlim_{\tiny 0< l,t\to \infty} {\img (H_r(\tilde M_{-l}\sqcup \tilde M^t)\to H_r(\tilde M))}= \mathbb I_{-\infty}(r)+ \mathbb I^{\infty}(r)= C_r(M).$$ In view of Proposition (\ref{P36}) \ $C_r(M)$ is equal to $T H_r(\tilde M) $ hence Item 3. 
holds too.   

\end{proof}
\vskip .1in

The diagram (\ref{E13}) and the above observations induce the  diagram (\ref{E14}) with first three horizontal arrows isomorphisms and the last  arrow ($PD^N$ ) injective and $\kappa[t^{-1},t]-$linear.

\begin{equation} \label {E14}
\xymatrix {\mathbb F^{\tilde f}_r(a,b)/ T(H_r(\tilde M)= ^{BM}\mathbb F^{\tilde f}_r(a,b)\ar[ddd]\ar[dr]\ar[rr]^{PD^{BM}_r(a,b)}&&(\mathbb G^{\tilde f}_{n-r}(b,a))^\ast\ar[dr]^{u^\ast _{n-r}}\ar[ddd]^{p_r^\ast}&&\\
&\mathbb F^{\tilde f}_r(B)\ar[d]^{\pi^{ab}_{B,r}}\ar[rr]^{PD_r(B)}&&(\mathbb G^{\tilde f}_{n-r}(B'))^\ast&\\
&\hat\delta^{\tilde f}_r(a,b)\ar[rr]^{\hat{PD_r}(a,b)}&&(\hat\delta^{\tilde f}_{n-r}(b,a))^\ast\\
H^{N}_r(M ;\xi)\ar[rr]^{PD^{N}_r}&&(H^N_{n-r}(M;\xi))^\ast &&}
\end{equation}
Indeed, the first three  horizontal arrows are isomorphisms in view of the isomorphism $\mathbb F^{\tilde f}_r(a,b)/ T(H_r(\tilde M)\simeq \  ^{BM}\mathbb F^{\tilde f}_r(a,b).$
 
 The bottom arrow is the composition 

\begin{equation}\label {EE21}
\begin{aligned}
H_r(\tilde M)/ TH_r(\tilde M)= H^{N}_r(M;\xi) \to H^{BM}_r(\tilde M)/ TH^{BM}_r(\tilde M) \to\\
\to   (H_{n-r}(\tilde M))^\ast / T(H_{n-r}(\tilde M)^\ast)\xleftarrow{=}
(H^N_{n-r}(M;\xi))^\ast
\end{aligned}
\end{equation}
with the first arrow 
$\kappa[t^{-1},t]-$linear and injective in view of Proposition \ref{P52} item 3. and the second arrow, $ H^{BM}_r(\tilde M)/ T(H^{BM}_r(\tilde M)) \to
 H_{n-r}(\tilde M)^\ast/ T(H_{n-r}(\tilde M)^\ast)$ a 
 $\kappa[t^{-1},t]-$linear  isomorphism in view  of the isomorphism $PD:\  ^{BM}H_r(\tilde M)\to H_{n-r}(\tilde M)^\ast$  and $\xleftarrow{=}$ 
 a canonical isomorphism. 
Indeed,  the finite dimensionality of $T(H_{n-r}(\tilde M))$ and the isomorphism $H_{n-r}(\tilde M)\simeq H^N_{n-r}( M;\xi) \oplus T(H_{n-r}(\tilde M))$ imply that the composition $H^N_r(M;\xi)^\ast \to H_{n-r}(\tilde M)^\ast  \to H_{n-r}(\tilde M)^\ast / T( H_{n-r}(\tilde M)^\ast)$ is a (canonical) isomorphism.  

\begin{obs}\label {OO}: 
\begin{enumerate}
\item the  diagram  
 \begin{equation}
\xymatrix{\hat \delta ^{\tilde f}_r(a,b) \ar[d]^{\hat t_r}\ar[rrr]^{\hat{PD}_r(a,b)} &&&(\hat \delta ^{\tilde f}_{n-r}r(b,a))^\ast\\ 
\hat \delta ^{\tilde f}_r(a+2\pi,b +2\pi) \ar[rrr]^{\hat{PD}_r(a+2\pi, b+2\pi)} &&&(\hat \delta ^{\tilde f}_{n-r}r(b+2\pi,a+2\pi))^\ast\ar[u]^{\hat t_r^\ast}}
\end{equation}
is commutative,
\item
 $\hat\delta^{\tilde f}(a,b)$ is a finite dimensional vector space and therefore
 $(\hat\delta^{\tilde f}(a,b))^\ast$ is a finite dimensional vector space isomorphic to $\hat\delta^{\tilde f}(a,b).$
\end{enumerate}
\end{obs}
Define 
$$\boxed{\hat \delta^f_r \langle a, b\rangle:= \oplus_{k\in \mathbb Z} \delta^{\tilde f}_r(a+2\pi k. b+2\pi k)}$$
which equipped with the isomorphism $\oplus_{k\in \mathbb Z} \hat t_r(a+2\pi k,b+2\pi k)$ is a free $\kappa[t^{-1},t]-$module and  
 $$\boxed{(\hat \delta^f_r)^\ast  \langle a, b\rangle:\oplus_{k\in \mathbb Z} (\delta^{\tilde f}_r(a+2\pi k. b+2\pi k))^\ast}$$
 which equipped with the isomorphism $\oplus_{k\in \mathbb Z} \hat t_r(a+2\pi k,b+2\pi k)$ is a free $\kappa[t^{-1},t]-$module.

Note that $(\hat \delta^f_r)^\ast  \langle a, b\rangle$ is not the same as  $(\hat \delta^f_r \langle a, b\rangle)^\ast$ actually $(\hat \delta^f_r)^\ast  \langle a, b\rangle \subseteq (\hat \delta^f_r  \langle a, b\rangle)^\ast$ the first is a a f.g. free $\kappa[t^{-1},t]$ module the second is in general infinitely generated  but in case  of equality which happens only in case  that $\hat \delta^f_r \langle a, b\rangle= 0.$

\vskip .1in
{\it Finalizing the proof  of Theorem \ref{TT3}}

 In view of Observation \ref{OO} item 2.  and Diagram \ref{E14}above we have the isomorphism of $\kappa[t^{-1},t]-$modules 
$$\boxed{\hat {PD}_r \langle a, b\rangle: \hat\delta^f_r\langle a,b\rangle \to  (\hat\delta^f_{n-r})^\ast\langle b,a\rangle}.$$ 

The choice of compatible splittings $"\mathcal S``$ provides an isomorphism of $\kappa[t^{-1},t]-$ modules  
$$\pi(r) \cdot  I^{\mathcal S}_r: \oplus _{\langle a, b\rangle} (\hat \delta^f_r)^\ast  \langle a, b\rangle \to H^N_r(M; \xi_f)$$  and then establishes the isomorphism of $H^N_r(X; \xi_f)$ to $H^N_{n-r}(M; \xi_f)$ which intertwines 
$(\hat \delta^f_r)^\ast  \langle a, b\rangle$ with $(\hat \delta^f_{n-r})^\ast  \langle a, b\rangle.$ 
This establishes items 1. and 2. in Theorem \ref {TT3}.

Suppose that $\kappa= \mathbb R$ or $\kappa= \mathbb R.$  Choose a non degenerate positive definite inner product on $H_r(\tilde M)$ which makes $t_r$ an isometry  for any $r.$ Such inner 
product can be provided by a Riemannian metric on $M$ when $M$ is a closed smooth manifold or by a triangulation of $M$ when $M$ is triangulable.
simply by lifting the metric or the triangulation on $\tilde M.$  

This inner products provide canonical compatible splittings which  realize canonically $\hat \delta^{\tilde f}(a,b)$ as subspaces of $H_r(\tilde M)$ and then of  $H^N_r(M;\xi_f)$ and lead to the embedding of $\hat \delta^f_r\langle a,b\rangle$ as a sub $\kappa[t^{-1},t]-$module of $H^N(X;\xi_f)$ (in view of the observation that the images of $\hat \delta^{\tilde f}(a,b)$and $\hat \delta^{\tilde f}(a+2\pi,b+2\pi)$ are orthogonal and $\hat t_r$ is an isometry). These canonical splittings  make 
$\pi(r) \cdot I^{\mathcal S}_r$ a canonical isomorphism. The inner products on $\hat\delta^{\tilde f}_r(a,b)$ induced from the inner product on $H_r(\tilde M)$   
canonically identifies  $\hat \delta^{\tilde f}_r(a,b)$  to $\hat \delta^{\tilde f}_r(a,b)^\ast$ then $\hat \delta^f_r\langle a,b\rangle$  to 
$(\hat \delta^f_r)^\ast\langle a,b\rangle$ and provides a canonical isomorphism
$\hat{PD}_r\langle a, b\rangle : \hat \delta^f_r\langle a,b\rangle  \to 
\hat \delta^f_{n-r} \langle a,b\rangle$ and then the isomorphism $\hat{PD}_r: H^N(M,\xi_f)\to 
H^N(M;\xi_f)_{n-r}.$ 

In case $\kappa= \mathbb C,$ 
if $\hat{\hat \delta}^f_r\langle a, b\rangle$ 
denotes the von Neumann completion  of $\hat \delta^f_r\langle a, b\rangle$  (note that $H^{L_2}(\tilde M)$ is the von Neumann completion  of $H^N(M;\xi_f)$) then $\hat{\hat \delta}^f\langle a, b\rangle$ is a closed Hilbert submodule of $H^{L_2}(\tilde M).$
Moreover
the  von Neumann completion  leads to the canonical isomorphism of $L^\infty(\mathbb S^1)-$Hilbert modules 
$$\hat {\hat {PD}}_r: H^{L_2}(\tilde M)\to H^{L_2}_{n-r}(\tilde M)$$ which intertwines 
 $\hat{\hat \delta}^f_r\langle a, b\rangle$ with $\hat{\hat \delta}^f_{n-r}\langle b, a\rangle.$
This establishes item 3. of Theorem \ref {TT3}.

\section {Proof of Observation \ref{C4} and Theorem \ref{T15}}\

{\it Proof of Observation \ref{C4}}
 
Suppose $X= X_1\cup X_2, Y= X_1\cap X_2$ with $X_1, X_2, Y$ closed subsets of $X$ with $X, X_1, X_2, Y$ all compact ANRs. Suppose $\xi \in H^1(X;\mathbb Z)$ and let $\xi_1, \xi_2, \xi_o$ be the pull backs of $\xi$ on $X_1, X_2, Y$ and let $\tilde X, \tilde X_1,\tilde X_2, \tilde Y$ be the infinite cyclic cover of $\xi, \xi_1, \xi_2, \xi_0. $ Note that 
$\tilde X= \tilde X_1\cup\tilde X_2$ $\tilde X_1\cap\tilde X_2= \tilde Y$ and then the long exact sequence in homology 
$$\xymatrix{\cdots\ar[r]&H_r(\tilde Y)\ar[r] & H_r(\tilde X_1)\oplus H_r(\tilde X_1) \ar[r]&H_r(\tilde X)\ar[r] &H_{r-1}(\tilde Y)\ar[r] &\cdots}$$    
is a sequence of $\kappa[t^{-1},t]-$modules with all arrows $\kappa[t^{-1},t]-$linear. If $H^N(Y;\xi_0)=0$ then $\beta^N_r(X_1;\xi_1)+ \beta^N_r(X_2;\xi_2)= \beta^N_r(X;\xi).$ 

We apply this to the double $X= DM= M_1\cup_{\partial M} M_2$ with $M_1$ equal to $M$ and $M_2$ equal to the  manifold $M$ with the opposite orientation  
and $\xi_D\in H^1(DM;\mathbb Z)$ a cohomology class which restricts to $\xi $ on $M_1$ and $M_2$ respectively. 
Clearly $\beta^N_r(DM;\xi_D)= 2 \beta^N_r(M; \xi).$  Since $DM$ is closed and orientable, consequently satisfies $\beta^N_r(DM;\xi_D)= \beta^N_{n-r}(DM; \xi_D),$ the statement follows.

q.e.d.
\vskip .1in
{\it Proof of Theorem \ref{T15}}

Items 1.  and 2. (a) follow from Observation \ref {C4} and the fact that both Betti numbers and Novikov--Betti numbers calculate the same Euler--Poincar\'e characteristic
cf \cite {F} or  \cite {Pa}.  

Items 1. and  2 (b)and (c)  follow from Proposition 4.1 in \cite {BH} which calculates $H_r(X;(\xi, u)).$

Item 3.  follows from  Theorem 1.4 in \cite{BH}. Indeed the hypotheses imply the existence of a tame map $f: M\to \mathbb S^1$ with a given angle $\theta$ regular value and $V= f^{-1}(\theta).$  Since the homology of $V$ is trivial in all dimensions but zero, the relation $R^\theta_r$ $r\ne 0$ is the trivial relation and  $R^\theta_0= id_\kappa.$  The statement follows from the description of Jordan cells in terms   of linear relations $R^\theta_r$ given by Theorem 1.4 in \cite{BH}.

As pointed out to us by L Maxim, the complement $X = \mathbb C^n\setminus V$ of a complex hyper surface $V\subset \mathbb C^n, V:= \{(z_1, z_2, \cdots z_n) \mid f(z_1, z_2, \cdots z_n)=0\}$ regular at infinity, 
equipped with the canonical class $\xi_f\in H^1(X:\mathbb Z)$ defined by $f: X\to \mathbb C\setminus 0$ is an example of an open manifold with an integral  cohomology class which has as compactification a manifold with boundary with a cohomology class which satisfies the hypotheses above.  

Item 1. recovers    
 a calculation of L Maxim, cf \cite{M14} and \cite {FM16} \footnote {The  Friedl-Maxim results  state the vanishing of more general and more sophisticated $L_2-$homologies and Novikov type  homologies. They can be also recovered via the appropriate Poincar\'e Duality isomorphisms}   
that the complement of an algebraic hyper surface regular at infinity   
has vanishing Novikov homologies in all dimension but $n.$

\section { Appendix (Poincar\'e duality for non closed manifolds derived from $\tilde M$)}  

Consistent with the previous notations let  $\tilde M(a,b)$ and $\tilde M(c)$ denote the compact set $\tilde f^{-1}([a,b])$ and $\tilde f^{-1}(a)$ which for $a,b,c $  regular
values are  sub manifolds of $\tilde M$ (the first, $M(a,b),$ is a  manifold with boundary $\partial M(a,b)= M(a)\sqcup M(b)$). We also recall that $\tilde M_a=\tilde f^{-1}((-\infty, a])$ and $\tilde M^b=\tilde f^{-1}([b,\infty)).$

Note that the Poincar\'e Duality for bordisms provides 
the isomorphisms  

\begin{equation}
\begin{aligned}
PD(-l,a)&: H_r(\tilde M(-l, a), \tilde M(-l)) \to H^{n-r}(\tilde M(-l, a), \tilde M(a)), \  -l < a \\ 
PD(b,t)&: H_r(\tilde M(b, t), \tilde M(t)) \to H^{n-r}(\tilde M(b, t), \tilde M(b)), \ t >b\\
PD(-l, t)&: H_r(\tilde M(-l, t), \tilde M(-l)\sqcup \tilde M(t)) \to H^{n-r}(\tilde M(-l, +t)), \ t,l>0 .\\
\end{aligned}
\end{equation}

Combining with excision property in homology or cohomology and 
passing to limit when $0<l\to \infty,$ and $0< l,t\to \infty$ one derives the Poincar\'e Duality isomorphisms 
 
\begin{equation}\label{EE1}
\begin{aligned}
PD^1_a: &H^{BM}_r(\tilde M_a)\to H^{n-r}(\tilde M, \tilde M^a)\\
PD^b_1: &H^{BM}_r(\tilde M^b)\to H^{n-r}(\tilde M, \tilde M_b)\\
PD : &H^{BM}_r(\tilde M)\to H^{n-r}(\tilde M)\\
PD^b_2: &H^{BM}_r(\tilde M, \tilde M^b)\to H^{n-r}(\tilde M_b)\\
PD^2_a: &H^{BM}_r(\tilde M, \tilde M_a)\to H^{n-r}(\tilde M^a)\\
\end{aligned}
\end{equation}

where\begin{equation}
\begin{aligned}
PD^1_a=\varprojlim_{l\to\infty} PD(-l, a), \ &\ 
PD_1^b=\varprojlim_{t\to\infty} PD(b,t)\\
PD= \varprojlim_{l\to\infty,
l\to \infty} PD(-l, t)\\
PD_2^b=\varprojlim_{l\to\infty, t=b} PD(-l, t), \ &\ 
PD_2^a=\varprojlim_{t\to\infty, -l=a} PD(-l, t) .
\end{aligned}
\end{equation}

These are the Poincar\'e Duality isomorphisms which appear in the diagrams (\ref{PD1})
and (\ref{PD1'}).

For example, in case of the first isomorphism in (\ref{EE1}),
\begin{equation*}
\begin{aligned}
H^{BM}(\tilde M_a)= &\varprojlim_{l\to \infty} (H_r(\tilde M(-l, a), \tilde M(-l)) \\
H^{n-r}(\tilde M, \tilde M^a)=H^{n-r}(\tilde M_a,\tilde M(a))= &\varprojlim_{l\to \infty} H^{n-r} (\tilde M(-l,a), \tilde M(a)) 
\end{aligned}\end{equation*}

where  the passage from $l$ to $l',$ \ $l'>l$, in the first equality above is derived from the commutative diagram
\vskip .1in

$\xymatrix{
H_r(\tilde M(-l,a), \tilde M(-l))\ar[r]^{=}\ar[d]&H_r(\tilde M(-l',a), \tilde M(-l',-l))&H_r(\tilde M(-l',a), \tilde M(-l'))\ar[l]\ar[d]\\ 
H^{n-r}(\tilde M(-l,a), \tilde M(a))&&H^{n-r}(\tilde M(-l',a), \tilde M(a))\ar[ll]}$

$\xleftarrow  {=}$

\end{document}